\documentclass{amsart}

\usepackage[english]{babel}
\usepackage{pgf}
\usepackage{tikz}
\usetikzlibrary{shapes,shapes.geometric,arrows,fit,calc,positioning,arrows,automata,}

\usepackage{hyperref}

\usepackage{amsfonts,amssymb,amsmath,amsthm}
\usepackage{url}
\usepackage{enumerate}

\newtheorem{theorem}{Theorem}[section]
\newtheorem{lemma}[theorem]{Lemma}
\newtheorem{proposition}[theorem]{Proposition}
\newtheorem{corollary}[theorem]{Corollary}
\theoremstyle{definition}
\newtheorem{definition}[theorem]{Definition}
\newtheorem{remark}[theorem]{Remark}
\numberwithin{equation}{section}

\usepackage[english]{cleveref}

\usepackage{array} 
\usepackage{tabularx}
\usepackage{graphicx} 

\usepackage{scrhack}
\usepackage{enumerate}

\allowdisplaybreaks

\newcommand{\N}{\mathbb{N}}
\newcommand{\Z}{\mathbb{Z}}
\newcommand{\R}{\mathbb{R}}

\newcommand{\C}{\mathbb{C}}
\renewcommand{\P}{\mathbb{P}}
\newcommand{\U}{\mathbb{U}}

\newcommand{\bfe}{\mathbf{e}}

\DeclareMathOperator{\e}{\mathrm{e}}

\DeclareMathOperator{\G}{\mathrm{G}}
\DeclareMathOperator{\T}{\mathrm{T}}

\DeclareMathOperator{\ind}{\mathbf{1}}

\newcommand{\floor}[1]{\left\lfloor #1 \right\rfloor}
\newcommand{\ceil}[1]{\left\lceil #1 \right\rceil}
\newcommand{\abs}[1]{\left| #1 \right|}
\newcommand{\norm}[1]{\left\| #1 \right\|}
\newcommand{\norminf}[1]{\left\| #1 \right\|_\infty}

\newcommand{\dv}{\, \mid \,}
\newcommand{\notdv}{\, \nmid \,}
\newcommand{\rb}[1]{\left( #1 \right)}

\usepackage{pdfpages}

\begin{document}
\selectlanguage{english}

\title[The Rudin-Shapiro sequence et al. are normal along squares]{The Rudin-Shapiro sequence and similar sequences are normal along squares}
\author{Clemens M\"ullner}
\address{Institut f\"ur Diskrete Mathematik und Geometrie TU Wien\\
Wiedner Hauptstr. 8-10, 1040 Wien, Austria}
\email{clemens.muellner@tuwien.ac.at}

\thanks{This research is supported by the project F55-02 of the Austrian Science Fund FWF which is part of the Special Research Program ``Quasi-Monte Carlo Methods: Theory and Applications'',
by Project F5002-N15 (FWF), which is a part of the Special Research Program “Algorithmic and Enumerative Combinatorics” and by Project I1751 (FWF), called MUDERA
(Multiplicativity, Determinism, and Randomness).}

\begin{abstract}
  We prove that digital sequences modulo $m$ along squares are normal,
  which covers some prominent sequences like the sum of digits in base $q$ modulo $m$, the Rudin-Shapiro sequence and some generalizations.
  This gives, for any base, a class of explicit normal numbers that can be efficiently generated.
\end{abstract}

\maketitle

\section{Introduction}
This paper deals with digital sequences modulo $m$.
Such sequences are ``simple'' in the sense that they are deterministic and uniformly recurrent sequences.
We show that the situation changes completely when we consider the subsequence along squares, i.e., we show that this subsequence is normal.\\
Thus, we describe a new class of normal numbers that can be efficiently generated, i.e., the first $n$ digits of the normal number can be generated
by using $O(n \log(n))$ elementary operations.

In this paper we let $\N$ denote the set of positive integers and we let $\mathbb{P}$ denote the set of prime numbers.
We let $\U$ denote the set of complex numbers of modulus $1$ and we use
the abbreviation $\e(x) = $ exp$(2\pi i x)$ for any real number $x$.\\
For two functions, $f$ and $g$ that take only strictly positive real values we write $f = O(g)$ or $f\ll g$ if $f/g$ is bounded.\\
We let $\floor{x}$ denote the floor function and $\{x\}$ denote the fractional part of $x$.
Furthermore, we let $\chi_{\alpha}(x)$ denote the indicator function for $\{x\}$ in $[0,\alpha)$.\\
Moreover we let $\tau(n)$ denote the number of divisors of $n$, $\omega(n)$ denote the number of distinct prime factors of $n$ and 
$\varphi(n)$ denote the number of positive integers smaller than $n$ that are co-prime to $n$.\\
Furthermore, let $\varepsilon_j^{(q)}(n) \in \{0,\ldots,q-1\}$ denote the $j$-th digit in the base $q$ expansion of a non-negative integer $n$, i.e.,
$n = \sum_{j=0}^{r} \varepsilon_j^{(q)}(n) q^j$, where $r = \floor{\log_q(n)}$.
We usually omit the superscript, as we work with arbitrary but fixed base $q \geq 2$.

\subsection{Digital Sequences}
The main topic of this paper are digital sequences modulo $m'$.
We use a slightly different definition of digital functions than the one found in~\cite{AlloucheShallit}.
\begin{definition}
  We call a function $b: \N \to \N$ a \emph{strongly block-additive $q$-ary function} or \emph{digital function} if there exist
  $m\in \N_{>0}$ and $F: \{0,\ldots,q-1\}^{m} \to \N$ such that $F(0,\ldots,0) = 0$ and
	\begin{align*}
		b(n) = \sum_{j\in\Z}	F(\varepsilon_{j+m-1}^{(q)}(n),\ldots,\varepsilon_{j}^{(q)}(n)),
	\end{align*}
\end{definition}
where we define $\varepsilon_{-j}(n) = 0$ for all $j\geq1$.
The difference to the usual definition is the range of the sum ($\N_0$ or $\Z$) which does not matter for all appearing examples.

\begin{remark}
  The name strongly block-additive $q$-ary function was inspired by (strongly) $q$-additive functions.
  Bellman and Shapiro~\cite{bellmanShapiro} and Gelfond~\cite{gelfond1968} denoted a function $f$ to be $q$-additive if
  \begin{align*}
    f(a q^r + b) = f(a q^r) + f(b)
  \end{align*}
  holds for all $r \geq 1$, $1\leq a < q$ and $0\leq b < q^r$.
  Mend\`es France~\cite{france} denoted a function $f$ to be strongly $q$-additive if
  \begin{align*}
    f(a q^r + b) = f(a) + f(b)
  \end{align*}
  holds for all $r \geq 1$, $1 \leq a < q$ and $0\leq b < q^r$.
  Thus, we can write for a strongly $q$-additive function $f$,
  \begin{align*}
    f(n) = \sum_{j\in \Z} f(\varepsilon_j^{(q)}(n)).
  \end{align*}
\end{remark}

A quite prominent example of a strongly block-additive function is the sum of digits function $s_q(n)$ in base $q$.
This is a strongly block-additive function with $m=1$ and $F(x) = x$.
In particular, $(s_2(n) \bmod 2)_{n\in\N}$ gives the well-known Thue--Morse sequence.\\
Another prominent example is the Rudin-Shapiro sequence $\mathbf{r} = (r_n)_{n\geq 0}$ which is given by the parity of the blocks of the form ``$11$'' in the digital expansion in base $2$.
Let $b$ be the digital sequence corresponding to
$q=2,m=2$ and $F(x,y) = x\cdot y$, then we find $r_n = (b(n) \bmod 2)$.
This can be generalized to functions that are given by the parity of blocks of the form ``$111\ldots11$'' for fixed length of the block; 
these functions have for example been mentioned and studied in~\cite{mauduit_rivat_rs}.

Digital sequences are regular sequences (see for example~\cite{Cateland}). 
Consequently we find that digital sequences modulo $m'$ are automatic sequences (see~\cite[Corollary 16.1.6]{AlloucheShallit})
which implies some interesting properties.
For a detailed treatment of automatic sequences see~\cite{AlloucheShallit}.

We define the subword complexity of a sequence $\mathbf{a}$, that takes only finitely many different values, as
  \begin{align*}
    p_{\mathbf{a}}(n) = \# \{(a_i,\ldots,a_{i+n-1}): i\geq0\}.
  \end{align*}
  It is well known that the subword complexity of automatic sequences is sub-linear (see \cite[Corollary 10.3.2]{AlloucheShallit}), i.e. for every automatic sequence $\mathbf{a}$ we have
  \begin{align*}
    p_{\mathbf{a}}(n) = O(n).
  \end{align*}
For a random sequence $\mathbf{u} \in \{0,1\}^{\N}$ one finds that $p_{\mathbf{u}}(n) = 2^{n}$ with probability one.
Thus, automatic sequences are far from being random.

\subsection{Main Result}
It is well known that these properties are preserved when considering arithmetic subsequences of automatic sequences and, therefore, digital sequences modulo $m'$.
However, the situation changes completely when one considers the subsequence along squares.
\begin{definition}
  A sequence $\mathbf{u} \in \{0,\ldots,m'-1\}^{\N}$ is normal if, for any $k \in \N$ and any $(c_0,\ldots,c_{k-1}) \in \{0,\ldots,m'-1\}^k$, we have
  \begin{align*}
    \lim_{N\to \infty} \frac{1}{N} \#\{i<N: u(i) = c_0,\ldots,u(i+k-1) = c_{k-1}\} = (m')^{-k}.
  \end{align*}
\end{definition}
Drmota, Mauduit and Rivat showed a first example for that phenomenon~\cite{drmotaMauduitRivat2014}. 
They considered the classical Thue--Morse sequence $(t_n)_{n\geq0}$ and showed that not only $p_{(t_{n^2})_{n\geq0}}(k) = 2^k$, 
but were able to show that $(t_{n^2})_{n\geq 0}$ is normal.
The fact that $p_{(t_{n^2})_{n\geq 0}}(k) = 2^k$ had already been proven by Moshe~\cite{moshe}, who was able to give exponentially growing lower bounds for extractions of the
Thue--Morse sequence along polynomials of degree at least $2$.
In this paper we go one step further than Drmota, Mauduit and Rivat and show a similar result for general digital sequences.

\begin{theorem}\label{maintheorem}
	Let $b$ be a digital function and $m'\in \N$ with $\gcd(q-1,m') = 1$ and\\ 
	$\gcd(m',\gcd(\{b(n): n\in \N\})) = 1$. 
	Then $(b(n^2)\bmod m')_{n\in \N}$ is normal.
\end{theorem}
There are only few known explicit constructions of normal numbers in a given base (see for example \cite[Chapters 4 and 5]{normal}).
This result provides us with a whole class of normal sequences for any given base that can be generated efficiently, 
i.e. it takes $O(n \log n)$ elementary operations to produce the first $n$ elements.\\
The easiest construction for normal sequences is the Champernowne construction that is given by concatenating the base $b$ expansion of successive integers.
This gives for example for base $10$: $123456789101112131415\ldots$.
Using the first $n'$ integers takes $O(n' \log(n'))$ elementary operations and gives a sequence of length $\Theta(n' \log(n'))$.\\
Scheerer~\cite{scheerer_normal} analyzed the runtime of some algorithms that produce absolutely normal numbers, 
i.e. real numbers in $[0,1]$ whose expansion in base $b$ are normal for every base $b$.
Algorithms by Sierpinski~\cite{Sierpinski_normal} and Turing~\cite{Turing_normal} use double exponentially many operations 
and algorithms by Levin~\cite{Levin_normal} and Schmidt~\cite{Schmidt_normal} use exponentially many operations.
Moreover, Becher, Heiber and Slaman~\cite{normal_polynomial} gave an algorithm that takes just above $n^2$ operations to produce the first $n$ digits.

Digital sequences modulo $m'$ have interesting (dynamical) properties.
Firstly, they are primitive and, therefore, uniformly recurrent (\cite[Theorem 10.9.5]{AlloucheShallit})
, i.e., every block that occurs in the sequence at least once, occurs infinitely often with bounded gaps.

There is a natural way to associate a dynamical system - the symbolic dynamical system - to a sequence that takes only finitely many values.
\begin{definition}
  The symbolic dynamical system associated to a sequence $\mathbf{u} \in \{0,\ldots,m'-1\}^{\N}$ is the system $(X(\mathbf{u}),T)$, 
  where $T$ is the shift on $\{0,\ldots,m'-1\}^{\N}$ and $X(\mathbf{u})$ the closure 
  of the orbit of $\mathbf{u}$ under the action of $T$ for the product topology of $\{0,\ldots, m'-1\}^{\N}$.
\end{definition}

Some of the mentioned properties of automatic sequences also imply important properties for the associated symbolic dynamical system.

The fact that every digital sequence modulo $m'$, denoted by $\mathbf{u}$, is uniformly recurrent implies that the associated symbolic dynamical system is minimal;
i.e., the only closed $T$ invariant sets in $X(\mathbf{u})$ are $\emptyset$ and $X(\mathbf{u})$ - see for example~\cite{subst} or \cite{queffelec}.

Furthermore, the entropy of symbolic dynamical systems to a sequence $\mathbf{u}$, that takes only finitely many values, is equal to
\begin{align*}
  \lim_{n\to\infty} \frac{\log(p_{\mathbf{u}}(n))}{n},
\end{align*}
(see for example~\cite{dynamics} or \cite{seq_complex}).
Consequently, we know that the entropy of the symbolic dynamical system associated to a digital sequence modulo $m'$ equals $0$, 
and, therefore, the dynamical system is deterministic.

\subsection{Outline of the proof}

In order to prove our main result, we will work with exponential sums. 
We present here the main theorem on exponential sums 
and further show its connection to Theorem \ref{maintheorem}.

\begin{theorem}\label{Thexponentialsums}
  For any integer $k\ge 1$ and $(\alpha_0,\ldots, \alpha_{k-1}) \in \{\frac0{m'},\ldots,\frac{m'-1}{m'}\}^k$
  such that $(\alpha_0,\ldots,\alpha_{k-1}) \ne (0,\ldots, 0)$,  there exists $\eta > 0$ such that
\begin{align}\label{eqThexponentialsums}
  S_0  = \sum_{n<N} \e\rb{\sum_{\ell=0}^{k-1} \alpha_\ell b((n+\ell)^2) } \ll N^{1-\eta}.
\end{align}
\end{theorem}

\begin{lemma}
Theorem~\ref{Thexponentialsums} implies Theorem \ref{maintheorem}.
\end{lemma}
\begin{proof}
Let $(c_0,\ldots,c_{k-1}) \in \{0,\ldots,m'-1\}^k$ be an arbitrary sequence of length $k$. 
We count the number of occurrences of this sequence in 
$(b(n^2)\bmod m')_{n\leq N}$. Assuming that (\ref{eqThexponentialsums}) holds, we obtain by using the well known identity  
$\sum_{n=0}^{m'-1} \e(\frac{n}{m'} \ell) = m'$ for $\ell \equiv 0 \bmod m'$ and $0$ otherwise
\begin{align*}
&\abs{ \{ n< N : (b(n^2)\bmod m', \ldots, b((n+k-1)^2)\bmod m') = (c_0,\ldots,c_{k-1})  \}} \\
&= \sum_{n< N} \ind_{[b(n^2) \equiv c_0\bmod m']} \cdots \ind_{[b((n+k-1)^2) \equiv c_{k-1}\bmod m']} \\
&= \sum_{n< N} \prod_{\ell = 0}^{k-1}\frac 1{m'} \sum_{\alpha_{\ell}' = 0}^{m'-1} 
\e\rb{ \frac{\alpha_{\ell}'}{m'} \rb{ b((n+\ell)^2) - c_{\ell} } }
\end{align*}
\begin{align*}
&= \frac 1{(m')^k} \sum_{\substack{ (\alpha_0',\ldots,\alpha_{k-1}')\\ \in \{0,\ldots,m'-1\}^k }} 
\e\rb{ -\frac {\alpha_0' c_0 + \cdots + \alpha_{k-1}'c_{k-1}}{m'}}
\sum_{n < N} \e\rb{ \sum_{\ell=0}^{k-1} \underbrace{\frac {\alpha_\ell'}{m'}}_{=:\alpha_{\ell}} b((n+\ell)^2) }\\
&= \frac{N}{(m')^k} + \mathcal{O}\rb{ N^{1-\eta} }
\end{align*}
with the same $\eta > 0$ as in Theorem \ref{Thexponentialsums}. \\
To obtain the last equality we separate the term with 
$(\alpha_0', \ldots, \alpha_{k-1}') = (0,\ldots,0)$.
\end{proof}

The structure of the 
rest of the paper is presented below.

In \Cref{sec:digital} we discuss some properties of digital sequences.
These properties will be very important for the estimates of the Fourier terms.

In Section \ref{cha:bounds}, we derive the main ingredients of the proof of Theorem~\ref{Thexponentialsums} which are upper bounds on the Fourier terms
\begin{displaymath}
H_\lambda^I(h,d) = \frac 1{q^\lambda} \sum_{0\le u < q^\lambda}
\e\rb{ \sum_{\ell=0}^{k-1} \alpha_\ell b_\lambda(u + \ell d + i_\ell) - h q^{-\lambda} },
\end{displaymath}
where $I = (i_0,\ldots, i_{k-1})\in\N^k$ with some special properties defined in \Cref{sec:Fourier_squares} and $b_{\lambda}$ is 
a truncated version of $b$ which is properly defined in \Cref{def:truncated_function}.

The main results of Section \ref{cha:bounds} are Propositions \ref{Pro1} and \ref{Pro2}. 
Proposition~\ref{Pro1} yields a bound on averages of Fourier transforms and 
Proposition~\ref{Pro2} yields a uniform bound on Fourier transforms.

In Section \ref{cha:proof}, we discuss how Proposition~\ref{Pro1} and Proposition~\ref{Pro2}
are used to prove Theorem \ref{Thexponentialsums}.
The approach is very similar to \cite{drmotaMauduitRivat2014} and we will mainly describe how it has to be adapted.
We use Van-der-Corput-like inequalities in order to reduce our problem to sums depending only on few digits of $n^2, (n+1)^2, \ldots, (n+k-1)^2$.
By detecting these few digits, we are able to remove the quadratic terms, which allows a proper
Fourier analytic treatment. After the Fourier analysis, the remaining sum is split into two sums. 
The first sum involves quadratic exponential sums which are dealt with using the results from \Cref{sec:gauss}.

The Fourier terms $H_\lambda^I(h,d)$ appear in the second sum and Propositions \ref{Pro1} and \ref{Pro2} provide the necessary bounds.
 
We have to distinguish the cases 
$K = \alpha_0 + \cdots + \alpha_{k-1} \in \Z$ and $K \notin \Z$.
Sections~\ref{sec:equiv0} and \ref{sec:nequiv0} tackle one of these cases each.
In Section~\ref{sec:equiv0}, we prove that~-- if $K \in \Z$~-- we deduce  
Theorem~\ref{Thexponentialsums} from  Proposition~\ref{Pro1}.
For $K \notin \Z$, Section~\ref{sec:nequiv0} shows that we can deduce  
Theorem~\ref{Thexponentialsums} from  Proposition~\ref{Pro2}.

In Section \ref{cha:auxiliary}, we present some auxiliary results also used in \cite{drmotaMauduitRivat2014}.

\section{Digital Functions}\label{sec:digital}
In this section we discuss some important properties of digital functions.
We start with some basic definitions.
\begin{definition}\label{def:truncated_function}
	We define for $0 \leq \mu \leq \lambda$ the truncated function $b_\lambda$ and the two-fold restricted function $b_{\mu,\lambda}$ by
  \begin{displaymath}
          b_\lambda(n) =  \sum_{j < \lambda} F(\varepsilon_{j+m-1}(n),\ldots,\varepsilon_{j}(n)) \text{ and } b_{\mu,\lambda}(n) = b_{\lambda}(n)-b_{\mu}(n).
  \end{displaymath}
\end{definition}
We see directly that $b_{\lambda}(.): \N\to\N$ is a $q^{\lambda+m-1}$ periodic function and we extend it to a ($q^{\lambda+m-1}$ periodic) function $\Z\to\N$ 
which we also denote by $b_{\lambda}(.): \Z\to\N$.

For any $n\in \N$, we define $F(n):=F(\varepsilon_{m-1}(n),\ldots,\varepsilon_{0}(n))$.

Since $F(0) = 0$, we can rewrite $b(n)$ and $b_{\lambda}(n)$ for $\lambda \geq 1$ as follows
\begin{align*}
  b(n) &= \sum_{j \geq 0} F\rb{\floor{\frac{q^{m-1}n}{q^j}}}\\
  b_{\lambda}(n) &= \sum_{j=0}^{\lambda+m-2} F\rb{\floor{\frac{q^{m-1}n}{q^j}}}.
\end{align*}

We show that for any block-additive function, we can choose $F$ without loss of generality such that it fulfills a nice property.
\begin{lemma}
  Let $b: \N \to \N$ be a strongly block-additive function corresponding to $F'$.
  Then, there exists another function $F$ such that $b$ also corresponds to $F$ and
  \begin{align}
    \label{property_b'2} \sum_{j = 1}^{m-1} F(nq^j) = 0
  \end{align}
  holds for all $n\in \N$.
\end{lemma}
\begin{proof}
  We start by defining a new function
  \begin{align*}
    G(n) := \sum_{j=1}^{m-1} F'(n q^j).
  \end{align*}
  This already allows us to define the function $F$:
  \begin{align*}
    F(n) := F'(n) + G(n) - G(\floor{n/q}).
  \end{align*}

  We find directly that $G(0) = F(0) = 0$. 
  It remains to show that $b$ corresponds to $F$ and that \eqref{property_b'2} holds, which are simple computations,
  \begin{align*}
    \sum_{j \geq 0} F\rb{\floor{\frac{q^{m-1}n}{q^j}}}&= \sum_{j \geq 0} F'\rb{\floor{\frac{q^{m-1}n}{q^j}}} \\
      &\qquad + \sum_{j \geq 0} G\rb{\floor{\frac{q^{m-1}n}{q^j}}} - \sum_{j \geq 0} G\rb{\floor{\frac{q^{m-1}n}{q^{j+1}}}}\\
      &= b(n) + G(0) = b(n).
  \end{align*}
  Furthermore, we find
  \begin{align*}
    \sum_{j = 1}^{m-1} F(nq^j) &= \sum_{j = 1}^{m-1} F'(nq^j) + \sum_{j = 1}^{m-1} G(nq^j) - \sum_{j = 1}^{m-1} G(nq^{j-1})\\
      &= \sum_{j = 1}^{m-1} F'(nq^j) + G(n q^{m-1}) - G(n)\\
      &= \sum_{j = 1}^{m-1} F'(nq^j) + 0 - \sum_{j=1}^{m-1} F'(n q^j) = 0.
  \end{align*}
\end{proof}
  
We assume from now on that for any strongly block-additive function $b$ \eqref{property_b'2} holds.
This allows us to find an easier expression for $b$:
\begin{corollary}
Let $b(n)$ be a digital function fulfilling \eqref{property_b'2}. 
Then 
\begin{align*}
  b(n) = \sum_{j \geq 0} F\rb{\floor{\frac{n}{q^j}}}
\end{align*}
and
\begin{align*}
  b_{\lambda}(n) = \sum_{j = 0}^{\lambda-1} F\rb{\floor{\frac{n}{q^j}}}
\end{align*}
holds for all $n, \lambda \in \N$.
\end{corollary}

We easily find the following recursion.
\begin{lemma}\label{le:rec_b}
Let $\alpha \in \N, n_1\in \N$ and $0\leq n_2 < q^{\alpha}$. Then
\begin{align}\label{eq:recursion_b_lambda}
	b_{\lambda}(n_1 q^{\alpha}+n_2) = b_{\lambda-\alpha}(n_1) + b_{\alpha}(n_1 q^{\alpha} + n_2)
\end{align}
holds for all $\lambda >\alpha$ and
\begin{align}\label{eq:recursion_b}
  b(n_1 q^{\alpha}+n_2) = b(n_1) + b_{\alpha}(n_1 q^{\alpha} + n_2).
\end{align}
\end{lemma}

\begin{proof}
We compute $b_{\lambda}(n_1q^{\alpha} + n_2)$
  \begin{align*}
    b_{\lambda}(n_1q^{\alpha}+n_2) &= \sum_{j=0}^{\lambda-1} F\rb{\floor{\frac{n_1 q^{\alpha} + n_2}{q^j}}}\\
    &= \sum_{j=\alpha}^{\lambda-1} F\rb{\floor{\frac{n_1 q^{\alpha} + n_2}{q^j}}} + \sum_{j=0}^{\alpha-1} F\rb{\floor{\frac{n_1 q^{\alpha} + n_2}{q^j}}}\\
    &= \sum_{j=0}^{\lambda-\alpha-1} F\rb{\floor{\frac{n_1}{q^j}}} + \sum_{j=0}^{\alpha-1} F\rb{\floor{\frac{n_1 q^{\alpha} + n_2}{q^j}}}\\
    &= b_{\lambda-\alpha}(n_1) + b_{\alpha}(n_1 q^{\alpha} + n_2).
  \end{align*}
  The second case can be treated analogously.
\end{proof}

As we are dealing with the distribution of digital functions along a special subsequence, we will start discussing some distributional result for digital functions.
\begin{lemma}\label{le:b_mod_m'}
  Let $b$ be a strongly block-additive function and $m'>1$.
  Then the following three statements are equivalent.
  \begin{enumerate}[$(i)$]
   \item $\exists n \in \N: m'\notdv  b(n)$
   \item $\exists n<q^m: m' \notdv  F(n)$
   \item $\exists n<q^m: m' \notdv  b(n)$
  \end{enumerate}
\end{lemma}
\begin{proof}
  Obviously $(iii) \implies (i)$.
  Next we show $(i) \implies (ii)$:\\
  Let $n_0$ be the smallest natural number $>0$ such that $m' \notdv b(n_0)$.
  By Lemma~\ref{le:rec_b} holds
  \begin{align*}
    b(n_0) = b(\floor{n_0/q}) + F(n_0).
  \end{align*}
  By the definition of $n_0$, holds $m' \dv b(\floor{n_0/q})$ and, therefore, $m' \notdv F(n_0) = F( n_0 \bmod q^m)$.
  
  It remains to prove $(ii) \implies (iii)$:\\
  Let $n_0$ be the smallest natural number $>0$ such that $m' \notdv F(n_0)$.
  By $(ii)$, we have $n_0<q^m$. 
  We compute $b(n_0) \bmod m'$,
  \begin{align*}
    b(n_0) = \sum_{j\geq 0} F\rb{\floor{\frac{n_0}{q^j}}} \equiv F(n_0) \not \equiv 0 (\bmod m')
  \end{align*}
  as $\floor{\frac{n_0}{q^j}}<n_0$ for $j\geq 1$ implies that $F\rb{\floor{\frac{n_0}{q^j}}}\equiv 0 (\bmod m')$.
\end{proof}

\begin{remark}
  The following example shows that we can not replace $m' \notdv .$ by $\gcd(m',.) = 1$ in Lemma~\ref{le:b_mod_m'}:\\
  Let $m = 1, q=3,m' = 6$ and $F(0)=0, F(1) = 2, F(2) = 3$.
  We see that $\gcd(m',F(n)) >1$ for all $n <q^m=3$ and also $\gcd(m',b(n))>1$ for all $n<q^m=3$.
  However, $b(5) = F(1) + F(2) = 5$ and $\gcd(m',b(5)) = 1$.
\end{remark}

Next, we show a technical result concerning block-additive functions, which will be useful later on.
\begin{lemma}\label{le:complicated_b_not_const}
  Let $b$ be a strongly block-additive function in base $q$ and $k>1$ such that $\gcd(k,q-1) = 1$ and $\gcd(k, \gcd(\{b(n): n \in \N\})) = 1$.
  Then there exist integers $\bfe_1,\bfe_2<q^{2m-1}$ such that
  \begin{align}
	\begin{split}
	\label{eq:b_not_const}
    &b(q^{m-1}(\bfe_1 + 1)-1) - b(q^{m-1}(\bfe_1+1)) \\
		\not \equiv\quad & b(q^{m-1}(\bfe_2 + 1)-1) - b(q^{m-1}(\bfe_2+1)) (\bmod k)
	\end{split}
  \end{align}
  holds.
\end{lemma}
\begin{proof}
  Without loss of generality we can restrict ourselves to the case $p \in \P$ where $p\dv k$.
  Let us assume on the contrary that there exists $c$ such that
  \begin{align*}
    b(q^{m-1}(\bfe +1)-1) - b(q^{m-1}(\bfe+1)) \equiv c (\bmod p)
  \end{align*}
  holds for all $\bfe < q^{2m-1}$.
  Under this assumption, we find a new expression for $b(n)\bmod p$, where $n<q^{m}$:
  \begin{align*}
    n \cdot q^{m-1} c &\equiv \sum_{\bfe<n q^{m-1}} \rb{b(q^{m-1}(\bfe+1)-1) - b(q^{m-1}(\bfe+1))}\\
      &\equiv \sum_{\bfe<n q^{m-1}} \rb{b(\bfe) + b_{m-1}(q^{m-1} \bfe + q^{m-1}-1) - b(\bfe+1)}\\
      &\equiv -b(n q^{m-1}) + \sum_{\bfe<nq^{m-1}} b_{m-1}(q^{m-1} \bfe + q^{m-1}-1)\\
      &\equiv -b(n q^{m-1}) + n \sum_{\bfe<q^{m-1}} b_{m-1}(q^{m-1} \bfe + q^{m-1}-1).
  \end{align*}
	The last equality holds since $b_{m-1}(q^{m-1}\bfe + q^{m-1}-1)$ is a $q^{m-1}$ periodic function in $\bfe$.
  This gives
  \begin{align}\label{eq:le_rest_b}
    b(n) = b(nq^{m-1}) \equiv n \rb{\sum_{\bfe<q^{m-1}} b_{m-1}(q^{m-1} \bfe + q^{m-1}-1) - q^{m-1} c} (\bmod p).
  \end{align}
  By comparing this expression for $b(1)$ and $b(q)$ - note that $b(1) = b(q)$ - we find
  \begin{align*}
    (q-1) \rb{\sum_{\bfe<q^{m-1}} b_{m-1}(q^{m-1} \bfe + q^{m-1}-1) - q^{m-1} c} \equiv 0 (\bmod p)\\
    \sum_{\bfe<q^{m-1}} b_{m-1}(q^{m-1} \bfe + q^{m-1}-1) - q^{m-1} c \equiv 0 (\bmod p)
  \end{align*}
  as $\gcd(p,q-1) = 1$.
  
  Together with \eqref{eq:le_rest_b}, this implies that $p\dv b(n)$ for all $n<q^{m}$.
  This is a contradiction to $\gcd(p, \gcd(\{b(n): n \in \N\})) = 1$ by Lemma~\ref{le:b_mod_m'}.
\end{proof}

We will use this result in a different form.

\begin{corollary}\label{co:diff_b}
  Let $b$ be a strongly block-additive function in base $q$ and $m'>1$ such that $\gcd(m',q-1) = 1$ and $\gcd(m', \gcd(\{b(n): n \in \N\})) = 1$.
  For every $\alpha \in \{\frac{1}{m'},\ldots,\frac{m'-1}{m'}\}$ exist $\bfe_1,\bfe_2<q^{2m-1}$ and $d\in \N$ such that $d\alpha \not \in \Z$ and
  \begin{align*}
    &b(q^{m-1}(\bfe_1 + 1)-1) - b(q^{m-1}(\bfe_1+1))\\
		& \qquad - b(q^{m-1}(\bfe_2 + 1)-1) + b(q^{m-1}(\bfe_2+1)) = d.
  \end{align*}
\end{corollary}
\begin{proof}
  Let $\alpha = \frac{x}{y}$ where $\gcd(x,y) = 1$ and $1<y\dv m'$.
  We apply Lemma~\ref{le:complicated_b_not_const} for $k = y$ and find $\bfe_1,\bfe_2$ such that
  \begin{align*}
    &b(q^{m-1}(\bfe_1 + 1)-1) - b(q^{m-1}(\bfe_1+1))\\
		&\qquad - b(q^{m-1}(\bfe_2 + 1)-1) + b(q^{m-1}(\bfe_2+1)) = d,
  \end{align*}
  where
  \begin{align*}
    d \not \equiv 0 (\bmod y).
  \end{align*}
  This implies
  \begin{align*}
    d \alpha = \frac{dx}{y} \not \equiv 0 (\bmod 1).
  \end{align*}
\end{proof}

\section{Bounds on Fourier Transforms} \label{cha:bounds}

The goal of this section is to prove Propositions~\ref{Pro1} and~\ref{Pro2}. To find the necessary bounds we first need to recall one important result 
on the norm of matrix products which was first presented by Drmota, Mauduit and Rivat~\cite{drmotaMauduitRivat2014}. 

Afterwards, we deal with Fourier estimates and formulate Proposition~\ref{Pro1} and Proposition~\ref{Pro2}. 
The following Sections~\ref{sec:Pro1} and~\ref{sec:Pro2} give proofs of Proposition~\ref{Pro1} and Proposition~\ref{Pro2}, respectively.

\subsection{Auxiliary Results for the Bounds of the Fourier Transforms}
In this section we state necessary conditions under which the product of matrices decreases exponentially with respect to the matrix row-sum norm.

\begin{lemma}
\label{le:matrixnorm}
Let ${\mathbf M}_\ell$, $\ell \in \N$, be $N\times N$-matrices with complex entries $M_{\ell;i,j}$, for ${1\le i,j\le N}$,
and absolute row sums
\begin{displaymath}
\sum_{j=1}^N |M_{\ell;i,j}| \le 1 \text{ for } 1 \leq i \leq N.
\end{displaymath}
Furthermore, we assume that there exist integers $m_0\ge 1$ and $m_1\ge 1$ and constants $c_0> 0$ and $\eta > 0$ such that
\begin{enumerate}
\item every product ${\mathbf A} = (A_{i,j})_{(i,j)\in\{1,\ldots,N\}^2}$ of $m_0$ consecutive matrices ${\mathbf M}_\ell$ has the property that,
\begin{align} \label{matrixNorm1}
  |A_{i,1}| \ge c_0 \quad \mbox{or}\quad \sum_{j=1}^N |A_{i,j}| \le 1-\eta \text{ for every row } i;
\end{align}
\item  every product ${\mathbf B} = (B_{i,j})_{(i,j)\in\{1,\ldots,N\}^2}$ of $m_1$ consecutive matrices ${\mathbf M}_\ell$ has the property
\begin{align} \label{matrixNorm2}
  \sum_{j=1}^N |B_{1,j}| \le 1-\eta.
\end{align}
\end{enumerate}

Then there exist constants $C> 0$ and $\delta> 0$ such that 
\begin{align}\label{eqLe0001}
\norminf{\prod_{\ell = r}^{r+k-1} {\mathbf M}_\ell } \le C q^{-\delta k}
\end{align}
uniformly for all $r\ge 0$ and $k\ge 0$ (where $\norminf{\cdot}$ denotes the matrix row-sum norm).
\end{lemma}

\begin{proof}
  See \cite{drmotaMauduitRivat2014}.
\end{proof}

\begin{lemma}\label{le:sum_4}
  Let $x_1,x_2,\xi_1,\xi_2 \in \R$. Then
  \begin{align*}
    \abs{\e(x_1) + \e(x_1+\xi_1)} + \abs{\e(x_2) + \e(x_2+\xi_2)} \leq 4 - 8 \rb{\sin\rb{\frac{\pi \norm{\xi_1-\xi_2}}{4}}}^2.
  \end{align*}

\end{lemma}
\begin{proof}
  The proof is a straight-forward computation and can be found for example at the end of the proof of \cite[Lemma 12]{mauduit_rivat_rs}.
\end{proof}

\subsection{Fourier estimates}\label{sec:Fourier_squares}

In this section, we discuss some general properties of the occurring Fourier terms. 

For any $k\in\N$, we denote by $\mathcal{I}_k$ the set of integer vectors $I = (i_0,\ldots,i_{k-1})$ with $i_0 < q^{m-1}$ and 
$i_{\ell-1} \leq i_{\ell} \leq i_{\ell-1} + q^{m-1}$ for $1\le \ell\le k-1$.\\
Furthermore, we denote by $\mathcal{I}'_k$ the set of integer vectors $I' = (i'_0,\ldots,i'_{k-1})$ with $i'_0 = 0$ and
$i'_{\ell-1} \leq i'_{\ell} \leq i'_{\ell-1} +1$.\\
This set $\mathcal{I}_k$ obviously consists of $q^{m-1}(q^{m-1}+1)^{k-1}$ elements. 
For any $I\in\mathcal{I}'_k$, $h\in\Z$ and $(d,\lambda)\in\N^2$, we define
\begin{align}\label{eq:def-H}
H_{\lambda}^{I}(h,d) = \frac 1{q^{\lambda+m-1}} \sum_{0\le u < q^{\lambda+m-1}} 
    \e \rb{ \sum_{\ell = 0}^{k-1} \alpha_\ell b_{\lambda}(u+\ell d + i_\ell) - hu q^{-\lambda-m+1} },
\end{align} 
for fixed coefficients $\alpha_\ell \in \{\frac0{m'},\ldots,\frac{m'-1}{m'}\}$. 
This sum $H_{\lambda}^{I}(\,.\,,d)$ can then be seen as the discrete Fourier transform 
of the function 
\begin{displaymath}
  u \mapsto \e\rb{ \sum_{\ell = 0}^{k-1} \alpha_\ell b_{\lambda}(u+\ell d + i_\ell) },
\end{displaymath}
which is $q^{\lambda+m-1}$ periodic.

Furthermore, we define the important parameter
\begin{displaymath}
  K := \alpha_0 + \cdots + \alpha_{k-1}.
\end{displaymath}

We would like to find a simple recursion for $H_{\lambda}$ in terms of $H_{\lambda-1}$.
Instead we relate it to a different function for which the recursion is much simpler:

\begin{align*}
	G_{\lambda}^{I} (h,d) = \frac{1}{q^{\lambda}} \sum_{u<q^{\lambda}} \e\rb{\sum_{\ell = 0}^{k-1} \alpha_{\ell} b_{\lambda}(q^{m-1}(u+\ell d) + i_{\ell})-huq^{-\lambda}}.
\end{align*}

This sum $G_{\lambda}^{I}(\,.\,,d)$ can then be seen as the discrete Fourier transform of the function
\begin{displaymath}
  u \mapsto \e\rb{\sum_{\ell = 0}^{k-1} \alpha_{\ell} b_{\lambda}(q^{m-1}(u+\ell d) + i_{\ell})},
\end{displaymath}
which is $q^{\lambda}$ periodic.
We show now how $G$ and $H$ are related.
\begin{lemma}\label{Le0}
  Let $I \in \mathcal{I}'_k, h \in \Z, (d,\lambda) \in \N^2$ and $\delta \in \{0,\ldots,q^{m-1}-1\}$. It holds
  \begin{small}
    \begin{align} \label{eq:H-recursion}
      H_{\lambda}^{I}(h,q^{m-1} d+\delta) = \frac{1}{q^{m-1}} \sum_{\varepsilon = 0}^{q^{m-1}-1} \e\rb{-\frac{h \varepsilon}{q^{\lambda+m-1}} }
      \G_{\lambda}^{J_{\varepsilon, \delta}}(h,d),
    \end{align}
  \end{small}
  where
  \begin{align*}
    J_{\varepsilon,\delta} = J_{\varepsilon,\delta}(I) = \rb{i_{\ell} + \ell \delta + \varepsilon}_{\ell \in \{0,\ldots,k-1\}} \in \mathcal{I}_k.
  \end{align*}
\end{lemma}
\begin{proof}
One checks easily that $J_{\varepsilon,\delta}(I)\in \mathcal{I}_k$.
We evaluate $H_{\lambda}^{I}(h,q^{m-1} d + \delta)$:

\begin{align*}
&H_{\lambda}^{I}(h,q^{m-1} d + \delta)\\
=&\frac 1{q^{\lambda+m-1}} 
\sum_{0\le u < q^{\lambda+m-1}} \e\rb{\sum_{\ell = 0}^{k-1} \alpha_\ell b_{\lambda}(u+\ell (q^{m-1} d + \delta) + i_\ell) - h u q^{-\lambda-m+1} }\\
=& \frac{1}{q^{\lambda+m-1}} \sum_{\varepsilon <q^{m-1}} \sum_{0 \leq u < q^{\lambda}} \e\rb{-\frac{h (q^{m-1} u)}{q^{\lambda+m-1}}} \e\rb{-\frac{h \varepsilon}{q^{\lambda+m-1}}} \\
&\qquad \cdot \e\rb{\sum_{\ell=0}^{k-1} \alpha_{\ell} b_{\lambda}(q^{m-1} u + \varepsilon + \ell (q^{m-1} d + \delta) +i_{\ell})}\\
=& \frac{1}{q^{\lambda+m-1}} \sum_{\varepsilon <q^{m-1}} \sum_{ u < q^{\lambda}} \e\rb{-\frac{h u}{q^{\lambda}}} \e\rb{-\frac{h \varepsilon}{q^{\lambda+m-1}}}\\
 &\qquad \qquad \cdot \e\rb{\sum_{\ell=0}^{k-1}\alpha_{\ell} b_{\lambda}\rb{(u + \ell d)q^{m-1} + (\ell \delta + i_{\ell} +\varepsilon)}}\\
=& \frac{1}{q^{m-1}} \sum_{\varepsilon <q^{m-1}} \e\rb{-\frac{h \varepsilon}{q^{\lambda+m-1}} }
	\G_{\lambda}^{J_{\varepsilon, \delta}}(h,d).
\end{align*}

\end{proof}

Next we define a transformation on $\mathcal{I}_k$ and a weight function $v$.
\begin{definition}
  Let $j\geq 1$ and $\varepsilon,\delta \in \{0,\ldots,q^{j}-1\}$.
  Then, we define for $I \in \mathcal{I}_k$
  \begin{align*}
    &T_{\varepsilon,\delta}^{j}(I) := \rb{\floor{\frac{i_{\ell} + q^{m-1}(\varepsilon + \ell \delta)}{q^{j}}}}_{\ell \in \{0,\ldots,k-1\}}\\
    &v^{j}(I,\varepsilon,\delta) := \e\rb{\sum_{\ell<k} \alpha_{\ell} \cdot b_{j}(i_{\ell} + q^{m-1}(\varepsilon + \ell \delta))}.
  \end{align*}
\end{definition}
We see immediately that $\abs{v^{j}(I,\varepsilon,\delta)} = 1$ for all possible values of $j,I,\varepsilon$ and $\delta$.
Furthermore, we extend the definition of $T^{j}$ for arbitrary $\varepsilon,\delta$ by
\begin{align*}
  T_{\varepsilon,\delta}^{j}(I) := T_{\varepsilon \bmod q^j,\delta \bmod q^j}^{j}(I).
\end{align*}

The next Lemma shows some basic properties of these functions.
\begin{lemma}
  Let $\lambda,j,j_1,j_2 \in \N$, $\varepsilon,\delta \in \{0,\ldots,q^{j}-1\}$ and $\varepsilon_i,\delta_i \in \{0,\ldots,q^{j_i}-1\}$. Then, the following facts hold.
  \begin{itemize}
    \item $T_{\varepsilon,\delta}^{j}(I) \in \mathcal{I}_{k}$\\
    \item $T_{\varepsilon_2, \delta_2}^{j_2} \circ T_{\varepsilon_1, \delta_1}^{j_1} = T_{\varepsilon_2 q^{j_1} + \varepsilon_1,\delta_2 q^{j_1} + \delta_1}^{j_1+j_2}$\\
    \item $G_{\lambda}^{I}(h,d) = \frac{1}{q^{\lambda}}\sum_{u<q^{\lambda}} v^{\lambda}(I,u,d) \e(-huq^{-\lambda})$.
  \end{itemize}
\end{lemma}
\begin{proof}
  The first two facts are direct consequences of basic properties of the floor function and the last fact is just a reformulation of the definition of $G$ in terms of $v$.
\end{proof}
Now we can find a nice recursion for the Fourier transform $G$.

\begin{lemma}\label{le:rec_G}
  Let $I \in \mathcal{I}_k, h \in \Z, d,\lambda \in \N$ and $1\leq j\leq \lambda, \delta \in \{0,\ldots,q^{j}-1\}$. We have
  \begin{align*}
    G_{\lambda}^{I}(h,q^jd+\delta) = \frac{1}{q^{j}} \sum_{\varepsilon<q^j} \e(-h\varepsilon q^{-\lambda})v^j(I,\varepsilon,\delta)\cdot G_{\lambda-j}^{T_{\varepsilon,\delta}^{j}(I)}(h,d).
  \end{align*}

\end{lemma}

\begin{proof}
	We evaluate $G_{\lambda}^{I} (h,q^j d + \delta)$ and use \eqref{eq:recursion_b_lambda}:
	\begin{small}
	\begin{align*}
		G_{\lambda}^{I} &(h,q^j d + \delta) = \frac{1}{q^{\lambda}} \sum_{u < q^{\lambda}} \e\rb{\sum_{\ell = 0}^{k-1} \alpha_{\ell} b_{\lambda}(q^{m-1}(u+\ell (q^j d + \delta)) + i_{\ell}) - h u q^{-\lambda}}\\
		&= \frac{1}{q^j} \sum_{\varepsilon<q^j} \frac{1}{q^{\lambda-j}} \sum_{u < q^{\lambda-j}} \e\rb{\sum_{\ell = 0}^{k-1} \alpha_{\ell} b_{\lambda} (q^{m-1+j} (u+\ell d) + q^{m-1} (\varepsilon + \ell \delta) + i_{\ell})}\\
		&\quad \cdot \e(-h(u q^j +\varepsilon)q^{-\lambda})\\
		&= \frac{1}{q^j} \sum_{\varepsilon<q^j} \e\rb{\sum_{\ell = 0}^{k-1} \alpha_{\ell} b_j(q^{m-1}(\varepsilon + \ell \delta) + i_{\ell})} \e(-h\varepsilon q^{-\lambda})\\
		&\quad \cdot \frac{1}{q^{\lambda-j}} \sum_{u<q^{\lambda-j}} \e\rb{\sum_{\ell = 0}^{k-1} \alpha_{\ell} b_{\lambda-j} \rb{q^{m-1}(u + \ell d) + \floor{\frac{\varepsilon q^{m-1} + \ell \delta q^{m-1} + i_{\ell}}{q^j}}} - h u q^{-\lambda+j}}\\
		&= \frac{1}{q^j} \sum_{\varepsilon<q^j} v^j(I,\varepsilon, \delta) \e(-h \varepsilon q^{-\lambda})\cdot G_{\lambda-j}^{T_{\varepsilon, \delta}^j(I)}(h,d).
	\end{align*}
	\end{small}
\end{proof}

The following propositions are crucial for our proof of the main Theorem~\ref{Thexponentialsums}. 

\begin{proposition}\label{Pro1}
  If $K \equiv 0 (\bmod 1)$ and $\frac 12 \lambda \le  \lambda' \le \lambda$, then there exists $\eta > 0$ such that
  for any $I\in \mathcal{I}'_k$
  \begin{displaymath}
    \frac 1{q^{\lambda'}} \sum_{0\le d < q^{\lambda'}} \abs{H_\lambda^{I}(h,d)}^2 \ll  q^{-\eta \lambda}
  \end{displaymath}
  holds uniformly for all integers $h$.
\end{proposition}

\begin{proposition}\label{Pro2}
  If $K \not \equiv 0 (\bmod 1)$, then there exists $\eta > 0$ such that for any $I\in \mathcal{I}'_k$ 
  \begin{displaymath}
    \abs{ H_\lambda^{I}(h,d) } \ll  q^{-\eta L} \max_{J \in \mathcal{I}_k} \abs{ G_{\lambda-L}^{J}(h,\floor{ d/q^L }) } 
  \end{displaymath}
  holds uniformly for all non-negative integers $h,d$ and $L$.
\end{proposition}
Proofs for Proposition~\ref{Pro1} and~\ref{Pro2} are given in the following sections.

\subsection{Proof of Proposition~\ref{Pro1}}\label{sec:Pro1}
This section is dedicated to prove Proposition~\ref{Pro1}. 
We start by reducing the problem from $H_{\lambda}^{I}(h,d)$ to $G_{\lambda}^{I}(h,d)$ for which we have found a nice recursion.
\begin{proposition}\label{pro1_new}
  For $K \in \Z$ and $\frac{1}{2} \lambda \le \lambda' \le \lambda$, we find $\eta > 0$ such that for any $I \in \mathcal{I}_k$
  \begin{align*}
   \frac{1}{q^{\lambda'}} \sum_{0 \le d< q^{\lambda'}} \abs{G_{\lambda}^{I}(h,d)}^2 \ll q^{-\eta \lambda}
  \end{align*}
  holds uniformly for all integers $h$.
\end{proposition}

\begin{lemma}
 Proposition \ref{pro1_new} implies Proposition \ref{Pro1}.
\end{lemma}
\begin{proof}
 We see by \eqref{eq:H-recursion} that
 \begin{align*}
    \abs{H_{\lambda}^{I}(h,d)}^2 \leq \max_{J\in \mathcal{I}_k} \abs{G_{\lambda}^{J}(h,\floor{d/q^{m-1}})}^2 
        \leq \sum_{J\in \mathcal{I}_k} \abs{G_{\lambda}^{J}(h,\floor{d/q^{m-1}})}^2.
 \end{align*}
 Thus we find
 \begin{align*}
    \frac{1}{q^{\lambda'}} \sum_{0 \le d< q^{\lambda'}} \abs{H_{\lambda}^{I}(h,d)}^2 \leq 
	  \sum_{J\in \mathcal{I}_k} \frac{1}{q^{\lambda'}} \sum_{0 \le d< q^{\lambda'}} \abs{G_{\lambda}^{J}(h,\floor{d/q^{m-1}})}^2
      \ll q^{-\eta \lambda}.
 \end{align*}
\end{proof}

Using Lemma~\ref{le:rec_G}, it is easy to establish a recursion for 
\begin{align*}
  \Phi_{\lambda, \lambda'}^{I, I'}(h) = \frac{1}{q^{\lambda'}} \sum_{0 \leq d < q^{\lambda'}} \G_{\lambda}^{I}(h,d) \overline{\G_{\lambda}^{I'}(h,d)},
\end{align*}
where $h\in\Z$, $(\lambda,\lambda')\in\N^2$ and  $(I,I')\in\mathcal{I}_k^2$.
For $\lambda,\lambda'\geq 1$ and $1\leq j\leq \min(\lambda,\lambda')$ it yields for $\Phi_{\lambda, \lambda'}^{I,I'}(h)$ the following expression
\begin{align*}
   \frac{1}{q^{3j}} \sum_{\delta <q^j} \sum_{\varepsilon_1 <q^j} \sum_{\varepsilon_2 <q^j} \e\rb{-\frac{(\varepsilon_1-\varepsilon_2)h}{q^{\lambda}}} 
	v^j(I,\varepsilon_1,\delta) \overline{v^j(I,\varepsilon_2,\delta)} \Phi_{\lambda-j,\lambda'-j}^{T^j_{\varepsilon_1, \delta}(I), T^j_{\varepsilon_2, \delta}(I')}(h).
\end{align*}
To find this recursion, one has to split up the sum over $0 \leq d < q^{\lambda'}$ into the equivalence classes modulo $q^j$. 

This identity gives rise to a vector recursion for $\Psi_{\lambda, \lambda'}(h) = \rb{\Phi_{\lambda, \lambda'}^{I,I'}(h)}_{(I,I')\in \mathcal{I}_k^2}$.
We use the recursion for $j=1$:
\begin{align*}
	\Psi_{\lambda, \lambda'}(h) = \mathbf{M}(h/q^{\lambda}) \cdot \Psi_{\lambda-1, \lambda'-1}(h)
\end{align*}
where the $2^{2(k-1)} \times 2^{2(k-1)}$-matrix $\mathbf{M}(\beta) = (M_{(I,I'),(J,J')}(\beta))_{((I,I'),(J,J')) \in \mathcal{I}_k^2 \times \mathcal{I}_k^2}$ 
is independent of $\lambda$ and $\lambda'$. By construction, all absolute row sums of $\textbf{M}(\beta)$ are bounded by $1$.

It is useful to interpret these matrices as weighted directed graphs.
The vertices are the pairs $(I,I') \in \mathcal{I}_k^2$ and, starting 
from each vertex, there are $q^3$ directed edges to the vertices $(\T_{\varepsilon_1, \delta}(I),\T_{\varepsilon_2, \delta}(I'))$
 - where $(\delta, \varepsilon_1,\varepsilon_2)\in\{0,\ldots,q-1\}^3$ - with corresponding weights 
\begin{align*}
\frac{1}{q^3} \e\rb{-\frac{(\varepsilon_1-\varepsilon_2)h}{q^{\lambda}}} v^1(I,\varepsilon_1,\delta) \overline{v^1(I',\varepsilon_2,\delta)}.
\end{align*}
Products of $j$ such matrices correspond to oriented paths of length $j$ in these graphs, which are weighted with the corresponding products. 
The entries at position $((I,I'),(J,J'))$ of such product matrices correspond to the sum of weights along paths from $(I,I')$ to $(J,J')$.
Lemma~\ref{le:rec_G} allows us to describe this product of matrices directly.

\begin{lemma}
  The entry $((I,I'),(J,J'))$ of $\mathbf{M}(h/q^{\lambda}) \cdot \mathbf{M}(h/q^{\lambda-1}) \cdot \ldots \cdot \mathbf{M}(h/q^{\lambda-j+1})$ equals
  \begin{align*}
    \frac{1}{q^{3j}} \sum_{\delta<q^j} \sum_{\varepsilon_1,\varepsilon_2 < q^j} \ind_{[T_{\varepsilon_1, \delta}^{j}(I) = J]} \ind_{[T_{\varepsilon_2, \delta}^{j}(I') = J']}
	    v^j(I,\varepsilon_1,\delta) \overline{v^j(I',\varepsilon_2,\delta)} \e\rb{-\frac{(\varepsilon_1 - \varepsilon_2)h}{q^{\lambda}}}.
  \end{align*}
\end{lemma}
\begin{proof}
  Follows directly by Lemma~\ref{le:rec_G}.
\end{proof}

This product of matrices corresponds to oriented paths of length $j$. 
They can be encoded by the triple $(\varepsilon_1, \varepsilon_2, \delta)$ and they correspond to a path from $(I,I')$ to 
$(T_{\varepsilon_1, \delta}^{j}(I), T_{\varepsilon_2, \delta}^{j}(I'))$ with unimodular weight 
$v^j(I,\varepsilon_1,\delta) \overline{v^j(I',\varepsilon_2,\delta)} \e\rb{-\frac{(\varepsilon_1 - \varepsilon_2)h}{q^{\lambda}}}$.

To simplify further computations we define
\begin{align*}
  n_{(I,I'),(J,J')}^{(j)} := \sum_{\delta<q^j} \sum_{\varepsilon_1,\varepsilon_2 < q^j} \ind_{[T_{\varepsilon_1, \delta}^{j}(I) = J]} \ind_{[T_{\varepsilon_2, \delta}^{j}(I') = J']}
\end{align*}
and find directly that
\begin{align*}
  \sum_{(J,J') \in \mathcal{I}_{k}^2} n_{(I,I'),(J,J')}^{(j)} = q^{3j}
\end{align*}
and the absolute value of the entry $((I,I'),(J,J'))$ of $\mathbf{M}(h/q^{\lambda}) \cdot \mathbf{M}(h/q^{\lambda-1}) \cdot \ldots \cdot \mathbf{M}(h/q^{\lambda-j+1})$
is bounded by $n_{(I,I'),(J,J')}^{(j)} q^{-3j}$.

In order to prove Proposition~\ref{Pro1}, we will use Lemma~\ref{le:matrixnorm} uniformly for $h$ with $\mathbf{M}_l = \mathbf{M}(h/q^l)$. 
Therefore, we need to check Conditions~\eqref{matrixNorm1} and~\eqref{matrixNorm2}. 

Note that, since $\frac{1}{2} \lambda \leq \lambda' \leq \lambda$, we have
\begin{align*}
  \Psi_{\lambda, \lambda'}(h) = \mathbf{M}(h/q^{\lambda}) \cdots \mathbf{M}(h/q^{\lambda-\lambda'+1}) \Psi_{\lambda-\lambda',0}(h).
\end{align*}

\begin{lemma}
  The matrices $M_{l}$ defined above fulfill Condition~\eqref{matrixNorm1} of Lemma~\ref{le:matrixnorm}.
\end{lemma}
\begin{proof}
We need to show that there exists an integer $m_0 \geq 1$ such that every product
\begin{align*}
\mathbf{A} = (A_{(I,I'),(J,J')})_{((I,I'),(J,J')) \in \mathcal{I}_k^2 \times \mathcal{I}_k^2}
\end{align*}
of $m_0$ consecutive matrices $\mathbf{M}_l=\mathbf{M}(h/q^l)$ verifies Condition~\eqref{matrixNorm1} of Lemma~\ref{le:matrixnorm}. 

We define $m_0 = m-1 + \ceil{\log_{q}(k+1)}$.
It follows directly from the definition, that $T_{0,0}^{m_0}(I) = \mathbf{0}$ 
for all $I \in \mathcal{I}_k$. 
In the graph interpretation this means that for every vertex $(I,I')$ there is a path of length $m_0$ from $(I,I')$ to $(\mathbf{0},\mathbf{0})$. 
Fix a row indexed by $(I,I')$ in the matrix $\mathbf{A}$. 
We already showed that the entry $A_{(I,I'),({\mathbf 0},{\mathbf 0})}$ is the sum of at least one term of absolute value $q^{-3 m_0}$, i.e., 
$n_{(I,I'),(\mathbf{0},\mathbf{0})}^{(m_0)} \geq 1$.

There are two possible cases. If the absolute row sum is at most
\begin{align*}
\le 1 - \eta
\end{align*}
with $\eta \leq q^{-3m_0}$ 
then we are done.

In case the absolute row sum is strictly greater than $1 - \eta$,
we show that $|A_{(I,I'),({\mathbf 0},{\mathbf 0})}| \ge q^{-3m_0}/2$:
The inequality $|A_{(I,I'),({\mathbf 0},{\mathbf 0})}| < q^{-3m_0}/2$ implies that $A_{(I,I'),({\mathbf 0},{\mathbf 0})}$
is the sum of at least two terms of absolute value $q^{-3m_0}$, i.e. $n_{(I,I'),(\mathbf{0},\mathbf{0})}^{(m_0)} \geq 2$. 
Thus, we can use the triangle inequality to bound the absolute row sum by
\begin{displaymath}
\sum_{(J,J')} | A_{(I,I'),(J,J')} |  \leq \abs{A_{(I,I'),(\mathbf{0},\mathbf{0})}} + q^{-3 m_0} \sum_{(J,J') \neq (\mathbf{0},\mathbf{0})} n_{(I,I'),(J,J')}^{(m_0)}.
\end{displaymath}
Since 
\begin{align*}
  \sum_{(J,J')} n_{(I,I'),(J,J')}^{(m_0)} = q^{3 m_0}
\end{align*}
we find
\begin{align*}
  \sum_{(J,J')} | A_{(I,I'),(J,J')} | &\leq \abs{A_{(I,I'),(\mathbf{0},\mathbf{0})}} + 1 - q^{-3 m_0} n_{(I,I'),(\mathbf{0},\mathbf{0})}^{(m_0)}\\
      &\leq q^{-3 m_0}/2 + 1 - 2 q^{-3 m_0} < 1 - q^{-3 m_0}.
\end{align*}

This contradicts the assumption that the absolute row sum is strictly greater than 
\begin{align*}
1 - \eta \geq 1-q^{-3m_0}.
\end{align*}
Consequently, we find
\begin{align*}
|A_{(I,I'),(\mathbf{0},\mathbf{0})}| \geq c_0 \text{ for } c_0 = q^{-3m_0}/2.
\end{align*}
\end{proof}

\begin{lemma}
	The matrices $M_{l}$ fulfill Condition~\eqref{matrixNorm2} of Lemma~\ref{le:matrixnorm}.
\end{lemma}
\begin{proof}
We need to show that there exists an integer $m_1\geq 1$ such 
that for every product
\begin{align*}
\mathbf{B}=(B_{(I,I'),(J,J')})_{((I,I'),(J,J'))\in\mathcal{I}_k^2\times\mathcal{I}_k^2}
\end{align*} 
of $m_1$ consecutive matrices
${\mathbf M}_l = {\mathbf M}(h/q^l)$ 
 the absolute row-sum of the first row is bounded by $1-\eta$.
We concentrate on the entry $B_{({\mathbf 0},{\mathbf 0}),({\mathbf 0},{\mathbf 0})}$,
i.e. we consider all possible paths from 
$({\mathbf 0},{\mathbf 0})$ to $({\mathbf 0},{\mathbf 0})$
of length $m_1$ in the corresponding graph and show that a positive
saving for the absolute row sum is just due to the structure of this entry.

Since $T_{00}^{m+\floor{\log_q(k)}}({\mathbf 0}) = T_{10}^{m+\floor{\log_q(k)}} ({\mathbf 0}) = {\mathbf 0}$, 
we have at least two paths from $(\mathbf{0},\mathbf{0})$ to $(\mathbf{0},\mathbf{0})$ and it follows 
that the entry $B_{({\mathbf 0},{\mathbf 0}),({\mathbf 0},{\mathbf 0})}$
is certainly a sum of $k_0 = k_0(m_1)\ge 2$ terms of absolute value $q^{-3m_1}$
(for every $m_1 \ge m + \floor{\log_q(k)}$). 
This means that there are $k_0\ge 2$ paths 
from  $({\mathbf 0},{\mathbf 0})$ to $({\mathbf 0},{\mathbf 0})$
of length $m_1$ in the corresponding graph, or in other words $n_{({\mathbf 0},{\mathbf 0}),({\mathbf 0},{\mathbf 0})}^{m_1} = k_0(m_1) \geq 2$

Our goal is to construct two paths $(\varepsilon_1^i, \varepsilon_2^i,\delta^i)$ from $(\mathbf{0}, \mathbf{0})$ to $(\mathbf{0},\mathbf{0})$ such that
\begin{align*}
  \abs{\sum_{i=1}^{2} v^{m_1}(\mathbf{0},\varepsilon_1^i,\delta^i) \overline{v^{m_1}(\mathbf{0},\varepsilon_2^i,\delta^i)} 
  \e\rb{-\frac{(\varepsilon_1^i - \varepsilon_2^i)h}{q^{\lambda}}}}
	\leq 2 - \eta
\end{align*}
holds for all $h\in \Z$.

We construct a path from $\mathbf 0$ to $(q^{m-1}-1,\ldots,q^{m-1}-1,q^{m-1},\ldots,q^{m-1}) =: I_0 \in \mathcal{I}_{k}$ with exactly $n_0+1$ times $q^{m-1}-1$
(where $n_0 = \min \{n \in \N: \alpha_n \neq 0\}$).
We set $n_1 = \floor{ \log_{q}(k)} + m$ and find the following lemma.

\begin{lemma}\label{le:pathI0}
Let $n_0,n_1$ and $I_0$ be as above.
Then 
\begin{align*}
  T^{n_1}_{q^{n_1}-n_0-1, 1}(\mathbf{0}) = I_0.
\end{align*}
\end{lemma}
\begin{proof}
  This follows directly by the definitions and simple computations.
\end{proof}

By applying Lemma~\ref{le:pathI0} we find a transformation from $\mathbf 0$ to $I_0$. 
This gives a path from $(\mathbf{0},\mathbf{0})$ to $(I_0,I_0)$ by applying this transformation component-wise.
We concatenate this path with another path $(\bfe_1, \bfe_2,0)$ of length $n_2 = 3m-1$ where $\bfe_i < q^{2m-1}$.
The weight of the concatenation of these two paths equals
\begin{align*}
  &v^{n_1}(\mathbf{0},q^{n_1}-n_0-1,1) v^{n_2}(I_0,\bfe_1,0)\\
	&\qquad \qquad \cdot \overline{v^{n_1}(\mathbf{0},q^{n_1}-n_0-1,1)} \overline{v^{n_2}(I_0,\bfe_2,0)} \e\rb{-\frac{(\bfe_1-\bfe_2)h}{q^{\lambda-n_1}}}\\
  &\qquad =  v^{n_2}(I_0,\bfe_1,0)\overline{v^{n_2}(I_0,\bfe_2,0)}\e\rb{-\frac{(\bfe_1-\bfe_2)h}{q^{\lambda-n_1}}}.
\end{align*}
We denote by $I_{0|\ell}$ the $\ell$-th coordinate of $I_0$ and see that
\begin{align*}
  T_{\bfe_i, 0}^{3m-1}(I_0) &= \rb{\floor{\frac{I_{0|\ell} + q^{m-1}\bfe_i}{q^{3m-1}}}}_{\ell \in \{0\ldots k-1\}}\\
      &\leq \rb{\floor{\frac{q^{m-1} + q^{m-1}(q^{2m-1}-1)}{q^{3m-1}}}}_{\ell \in \{0\ldots k-1\}} \\
      &= \rb{\floor{\frac{q^{m-1} \cdot q^{2m-1}}{q^{3m-1}}}}_{\ell \in \{0\ldots k-1\}}= \mathbf{0}
\end{align*}

Thus, we have found for each $\bfe_1,\bfe_2 < q^{2m-1}$ a path from $(\mathbf{0},\mathbf{0})$ to $(\mathbf{0},\mathbf{0})$.

We can use the special structure of $I_0$ to make the weight of this path more explicit:
At first, we note that
\begin{align*}
  \sum_{\ell = 0}^{n_0} \alpha_{\ell} = \alpha_{n_0}
\end{align*}
by the definition of $n_0$.
Furthermore, we use the condition $K = \sum_{\ell} \alpha_{\ell} \in \Z$ to find
\begin{align*}
  \sum_{\ell = n_0+1}^{k-1} \alpha_{\ell} \equiv -\alpha_{n_0} (\bmod 1).
\end{align*}

We find by the definition of $v$ that for each $\bfe<q^{2m-1}$,
\begin{align*}
  v^{3m-1}(I_0,\bfe,0) &= \e\rb{\sum_{\ell = 0}^{k-1} \alpha_{\ell} b_{3m-1}(q^{m-1} \bfe + I_{0|\ell})}\\
  &= \e\rb{\alpha_{n_0} \rb{b_{3m-1}(q^{m-1} \bfe + q^{m-1}-1) - b_{3m-1}(q^{m-1}\bfe + q^{m-1})}}\\
  &= \e\rb{\alpha_{n_0} \rb{b(q^{m-1} \bfe + q^{m-1}-1) - b(q^{m-1}(\bfe + 1)}}.
\end{align*}

We find by Corollary~\ref{co:diff_b} that there exist $\bfe_1,\bfe_2<q^{2m-1}$ such that
\begin{align*}
  &b(q^{m-1}(\bfe_1 + 1)-1) - b(q^{m-1}(\bfe_1+1))\\
	&\qquad - b(q^{m-1}(\bfe_2 + 1)-1) + b(q^{m-1}(\bfe_2+1)) = d
\end{align*}
and $\alpha_{n_0} d \not \in \Z$.

We now compare the following two paths from $(\mathbf{0},\mathbf{0})$ to $(\mathbf{0},\mathbf{0})$ of length $m_1 = n_1+n_2 = \floor{\log_q(k)} + 4m-1$:\\
\begin{itemize}
 \item $(\bfe_1q^{n_1} + q^{n_1}-n_0-1,\bfe_2q^{n_1} + q^{n_1}-n_0-1,1)$:
    We split up this path into the path of length $n_1$ from $(\mathbf{0},\mathbf{0})$ to $(I_0,I_0)$ and the path of length $n_2$ from $(I_0,I_0)$ to $(\mathbf{0},\mathbf{0})$:
    The first path can be described by the triple $(q^{n_1}-n_0-1, q^{n_1}-n_0-1, 1)$ and its weight is obviously $1$.\\
    The second path - i.e. the path from $(I_0,I_0)$ to $(\mathbf{0},\mathbf{0})$ - can be described by the triple $(\bfe_1,\bfe_2,0)$ and its weight equals
    \begin{align*}
      &v^{n_2}(I_0,\bfe_1,0) \overline{v^{n_2}(I_0,\bfe_2,0)} \e\rb{-\frac{(\bfe_1-\bfe_2)h}{q^{\lambda-n_1}}} \\
      &\qquad= \e\rb{\alpha_{n_0} \rb{b(q^{m-1} (\bfe_1+1)-1) - b(q^{m-1}(\bfe_1 + 1))}} \\
      &\qquad \qquad \overline{\e\rb{\alpha_{n_0} \rb{b(q^{m-1} (\bfe_2+1)-1) - b(q^{m-1}(\bfe_2 + 1)}}}\e\rb{-\frac{(\bfe_1-\bfe_2)h}{q^{\lambda-n_1}}}\\
      &\qquad = \e(\alpha_{n_0} d) \e\rb{-\frac{(\bfe_1-\bfe_2)h}{q^{\lambda-n_1}}}.
    \end{align*}
    Thus, the overall weight of the path from $(\mathbf{0},\mathbf{0})$ to $(\mathbf{0},\mathbf{0})$ has weight
    \begin{align*}
      \e(\alpha_{n_0} d) \e\rb{-\frac{(\bfe_1-\bfe_2)h}{q^{\lambda-n_1}}}.
    \end{align*}
  \item $(\bfe_1 q^{n_1}, \bfe_2 q^{n_1}, 0)$: we compute directly the weight of this path:
    \begin{align*}
      &v^{m_1}(\mathbf{0},\bfe_1 q^{n_1},0) \overline{v^{m_1}(\mathbf{0},\bfe_2 q^{n_1},0)} \e\rb{-\frac{(\bfe_1-\bfe_2)h}{q^{\lambda-n_1}}}\\
      &\qquad = \e\rb{\sum_{\ell=0}^{k-1} \alpha_{\ell} b_{m_1}(\bfe_1 q^{n_1}) - \sum_{\ell=0}^{k-1} \alpha_{\ell} b_{m_1}(\bfe_2 q^{n_1})} \e\rb{-\frac{(\bfe_1-\bfe_2)h}{q^{\lambda-n_1}}}\\
      &\qquad = \e\rb{K (b_{m_1}(\bfe_1 q^{n_1}) - b_{m_1}(\bfe_2 q^{n_1}))} \e\rb{-\frac{(\bfe_1-\bfe_2)h}{q^{\lambda-n_1}}}\\
      &\qquad = \e\rb{-\frac{(\bfe_1-\bfe_2)h}{q^{\lambda-n_1}}}.
    \end{align*}
\end{itemize}

We recall quickly that $\alpha_{\ell} \in \{\frac{0}{m'},\ldots,\frac{m'-1}{m'}\}$ for all $\ell \in \{0,\ldots,k-1\}$ and, therefore, also
$\alpha_{n_0} \in \{\frac{0}{m'},\ldots,\frac{m'-1}{m'}\}$.
We finally see that

\begin{align*}
|B_{({\mathbf 0},{\mathbf 0}),({\mathbf 0},{\mathbf 0})}| &\le \rb{k_0-2+\abs{\e(\alpha_{n_0} d) \e\rb{-\frac{(\bfe_1-\bfe_2)h}{q^{\lambda-n_1}}} + \e\rb{-\frac{(\bfe_1-\bfe_2)h}{q^{\lambda-n_1}}}}} q^{-3m_1}\\
&= (k_0-2 +|1+\e(\alpha_{n_0} d)|)q^{-3m_1} \\
&= (k_0-2 +2\abs{\cos\rb{\pi \alpha_{n_0}d}})q^{-3m_1} \\
&= \rb{k_0-2 +2\abs{1-2\rb{\sin\rb{\frac{\pi \alpha_{n_0}d}{2}}}^2}}q^{-3m_1} \\
&\leq \rb{k_0 - 4 \rb{\sin\rb{\frac{\pi}{2m'}}}^2}q^{-3m_1}.
\end{align*}

Thus we have
\begin{align*}
\sum_{(J,J')} | B_{(0,0),(J,J')} |  
&\leq 
\rb{k_0-4 \rb{\sin\rb{\frac{\pi}{2m'}}}^2} q^{-3m_1} + (1-k_0 q^{-3m_1})  \\
&\leq
1 - 4 \rb{\sin\rb{\frac{\pi}{2m'}}}^2 \cdot q^{-3m_1}  .
\end{align*}
Therefore condition~\eqref{matrixNorm2} of Lemma~\ref{le:matrixnorm}
is verified with $m_1 = \floor{\log_q(k)} + 4m-1$ and $\eta = 4 \rb{\sin\rb{\frac{\pi}{2m'}}}^2 q^{-3m_1}\geq 4\rb{\sin\rb{\frac{\pi}{2m'}}}^2 k^{-3} q^{-12m+3}>0$.
\end{proof}

At the end of this section, we want to recall the important steps of the proof of Proposition~\ref{Pro1}.
At first we observe that
\begin{align*}
	\frac{1}{q^{\lambda'}} \sum_{0\leq d < q^{\lambda'}}|G_{\lambda}^{I}(h,d)|^2 = \Phi_{\lambda,\lambda'}^{I,I}(h).
\end{align*}
 Thus Proposition~\ref{Pro1} is equivalent to $\Phi_{\lambda,\lambda'}^{I,I}(h) \ll q^{-\eta \lambda}$. 
 Next we considered the vector $\Psi_{\lambda,\lambda'}(h) = \rb{\Phi_{\lambda,\lambda'}^{I,I'}(h)}_{(I,I') \in \mathcal{I}_k^2}$ and find the recursion
\begin{align*}
	\Psi_{\lambda,\lambda'}(h) = \mathbf{M}(h/q^{\lambda}) \cdots \mathbf{M}(h/q^{\lambda-\lambda'+1}) \Psi_{\lambda-\lambda',0}(h)
\end{align*}
Then we defined $\mathbf{M}_{\ell} := \mathbf{M}(h/q^{\ell})$ and showed that we can apply Lemma~\ref{le:matrixnorm}. 
Therefore we know that~-- since $\abs{\Phi_{\lambda-\lambda'+1,0}^{I,I'}(h)} \leq 1$
\begin{align*}
	|\Phi_{\lambda,\lambda'}^{I,I'}(h)| \leq \norminf{\mathbf{M}_{\lambda}\cdots \mathbf{M}_{\lambda-\lambda'+1}} \leq C q^{-\delta \lambda'} \leq C q^{- \delta \lambda/2}
\end{align*}
with $C$ and $\delta$ obtained by Lemma~\ref{le:matrixnorm}. Thus we know that $\Phi_{\lambda,\lambda'}^{I,I'}(h) \ll q^{-\eta \lambda}$ with $\eta = \delta/2$ uniformly for all $h$.
This concludes the proof of Proposition~\ref{Pro1}.

\subsection{Proof of Proposition \ref{Pro2}}\label{sec:Pro2}

We start again by reducing the problem from $H_{\lambda'}^{I'}(h,d)$ to $G_{\lambda}^{I}(h,d)$, for possibly different values of $\lambda,\lambda'$ and $I,I'$.
\begin{proposition}\label{Pro2_new}
  For $K \not \equiv 0 (\bmod 1)$ there exists $\eta>0$ such that for any $I \in \mathcal{I}_{k}$
  \begin{align}
    \abs{G_{\lambda}^{I}(h,d)} \ll q^{-\eta L} \max_{J \in \mathcal{I}_k} \abs{G_{\lambda-L}^{J}(h,\floor{d/q^{L}})}
  \end{align}
  holds uniformly for all non-negative integers $h,d$ and $L$.
\end{proposition}
\begin{lemma}
  Proposition~\ref{Pro2_new} implies Proposition~\ref{Pro2}.
\end{lemma}
\begin{proof}
  Follows directly by \eqref{eq:H-recursion}.
\end{proof}
We assume from now on that $K \notin \Z$ holds.

We formulate Lemma~\ref{le:rec_G} as a matrix vector multiplication:
\begin{align*}
	G_{\lambda}(h,q^jd+\delta) = \frac{1}{q^j} M^{j}_{\delta}\rb{\e\rb{-\frac{h}{q^{\lambda}}}}G_{\lambda-j}\rb{h,d}
\end{align*}
where for any $\delta \in \{0,\ldots, q^j-1\}$ and $z \in \U$ we have
\begin{align*}
	M^{j}_{\delta}(z) = \sum_{\varepsilon = 0}^{q^j-1}(\mathbf{1}_{[J = T^{j}_{\varepsilon, \delta}(I)]} v^j(I,\varepsilon,\delta) z^{\varepsilon})_{(I,J) \in \mathcal{I}_k^2}.
\end{align*}

To prove Proposition~\ref{Pro2_new} we aim to show that
\begin{align}\label{eq:M_saving}
  \exists m_1\in \N,\eta' \in \R^+ \text{ such that } \forall \delta<q^{m_1}, z\in \U \text{ holds } \norm{M^{m_1}_{\delta}(z)}_{\infty} \leq q^{m_1}-\eta'.
\end{align}

Indeed, we find that this is already sufficient to show Proposition~\ref{Pro2_new}.
\begin{lemma}
  \eqref{eq:M_saving} implies Proposition~\ref{Pro2_new}.
\end{lemma}
\begin{proof}
  We first note that
  \begin{align*}
    \norm{M^{j}_{\delta}(z)}_{\infty} \leq q^{j}
  \end{align*}
  holds for all $z \in \U$, $j\in \N$ and $\delta <q^j$ by definition.\\
  Next we split the digital expansion of $d\bmod q^{L}$ - read from left to right - into $\floor{L/m_1}$ parts of length $m_1$ and possible one part of length $L\bmod m_1$.
  We denote the first parts by $\delta_1,\ldots,\delta_{\floor{L/m_1}}$ and the last part by $\delta_0$, i.e.,
  \begin{align*}
    d = q^{L \bmod m_1}\rb{\sum_{j=1}^{\floor{L/m_1}} \delta_j \cdot q^{\floor{L/m_1}-j}} + \delta_0.
  \end{align*}

  Thus we find
  \begin{align*}
    \max_{I\in \mathcal{I}_k} &\abs{G_{\lambda}^{I}(h,d)} = \norm{G_{\lambda}(h,d)}_{\infty}\\
      &\leq \frac{1}{q^{L}}\max_{z\in \U}\norm{M^{L}_{d}(z)}_{\infty} \cdot \norm{G_{\lambda-L}(h,\floor{d/q^{L}})}_{\infty}\\
      &\leq \frac{1}{q^{L}}\prod_{j=1}^{\floor{L/m_1}} \max_{z\in\U}\norm{M^{m_1}_{\delta_j}(z q^{m_1(j-1)})}_{\infty} \cdot q^{(L\bmod m_1)} \cdot \norm{G_{\lambda-L}(h,\floor{d/q^{L}})}_{\infty}\\
      &\leq \frac{1}{q^{L}}(q^{m_1}-\eta')^{\floor{L/m_1}} q^{(L\bmod m_1)} \cdot \norm{G_{\lambda-L}(h,\floor{d/q^{L}})}_{\infty}\\
      &\ll q^{-L\eta}\cdot \norm{G_{\lambda-L}(h,\floor{d/q^{L}})}_{\infty}
  \end{align*}
  where $\eta = \frac{\eta'}{q^{m_1}\log(q^{m_1})}>0$.
\end{proof}

Throughout the rest of this section, we aim to prove \eqref{eq:M_saving}.

Therefore, we try to find for each $I\in \mathcal{I}_{k}$ and $\delta<q^{m_1}$ a pair $(\varepsilon_1,\varepsilon_2)$ and $m_1'\leq m_1$ such that for all $z\in \U$ holds
\begin{align}\label{eq:goal_eps_12}
  \begin{split}
  &T_{\varepsilon_i,\delta}^{m_1'}(I) = T_{\varepsilon_i+1,\delta}^{m_1'}(I),\\
  &\abs{v^{m_1'}(I,\varepsilon_1,\delta) + z v^{m_1'}(I,\varepsilon_1+1,\delta)} + \abs{v^{m_1'}(I,\varepsilon_2,\delta) + z v^{m_1'}(I,\varepsilon_2+1,\delta)} \leq 4-\eta'.
  \end{split}
\end{align}
Let us assume for now that \eqref{eq:goal_eps_12} holds.
Indeed we find
  \begin{align*}
    \norm{M^{m_1'}_{\delta}(z)}_{\infty} &= \max_{I\in\mathcal{I}_{k}} \max_{z\in \U} \sum_{J\in\mathcal{I}_k} \abs{\sum_{\varepsilon<q^{m_1'}} \ind_{[T_{\varepsilon,\delta}^{m_1'}(I)=J]} z^{\varepsilon} v^{m_1'}(I,\varepsilon,\delta)}
  \end{align*}
  However, we find for each $I$ some $\varepsilon_1,\varepsilon_2$ fulfilling \eqref{eq:goal_eps_12}.
  This gives
  \begin{align*}
    &\max_{z\in \U} \sum_{J\in\mathcal{I}_k} \abs{\sum_{\varepsilon<q^{m_1'}} \ind_{[T_{\varepsilon,\delta}^{m_1'}(I)=J]} z^{\varepsilon} v^{m_1'}(I,\varepsilon,\delta)}\\
	    &\qquad \leq \rb{q^{m_1'}-4} + \sum_{i=1}^{2} \abs{\sum_{j=0}^{1}z^{\varepsilon_i+j} v^{m'_1}(I,\varepsilon_i+j,\delta) }\\
	    &\qquad \leq q^{m'_1} - \eta'.
  \end{align*}

  Thus, we find in total
  \begin{align*}
    \norm{M^{m_1}_{\delta}(z)}_{\infty} \leq q^{m_1-m'_1} (q^{m_1'}-\eta') \leq q^{m_1} - \eta'.
  \end{align*}
  
It just remains to find $\varepsilon_1,\varepsilon_2,m'_1$ fulfilling \eqref{eq:goal_eps_12} and this turns out to be a rather tricky task.

We fix now some arbitrary $I\in \mathcal{I}_{k}$ and $\delta \in \N$.
We start by defining for $0\leq x \leq (4m-2)k$ and $c \in \N$
\begin{align*}
    M_{x,c} = M_{x,(c\bmod q^{x})}:= \{\ell <k : \floor{i_{\ell}/q^{m-1}} + d \ell \equiv c (\bmod q^{x})\}
  \end{align*}
and show some basic properties of $M_{x,c}$.

\begin{lemma}\label{le:c_0}
  For every $x<q^{(4m-2)k}$ exists $c_0$ such that
  \begin{align*}
    \sum_{\ell \in M_{x,c_0}} \alpha_{\ell} \not \in \Z.
  \end{align*}
\end{lemma}
\begin{proof}
One finds easily that
  \begin{align*}
    \{0,\ldots,k-1\} = \bigcup_{c<q^{x}} M_{x,c},
  \end{align*}
  which means that $\{M_{x,c}: c<q^{x}\}$ is a partition of $\{0,\ldots,k-1\}$ for each $x$.
  Thus, we find for every $x$ 
  \begin{align*}
    \sum_{c} \sum_{\ell \in M_{x,c}} \alpha_{\ell} = \sum_{\ell<k} \alpha_{\ell} = K \not \in \Z
  \end{align*}
  and the proof follows easily.
\end{proof}

\begin{lemma}\label{le:x_0}
  Let $d<q^{(4m-2)k}$ and $I \in \mathcal{I}_k$.
  
  Then, there exists $0\leq x_0\leq (4m-2)(k-1)$ such that for each $c<q^{x_0}$ exists $c^{+} < q^{x_0+(4m-2)}$ such that
  \begin{align*}
      M_{x_0,c} = M_{x_0+(4m-2),c^{+}}.
  \end{align*}
\end{lemma}
\begin{remark}
  This is equivalent to the statement that
  \begin{align*}
    \floor{i_{\ell_1}/q^{m-1}} + d \ell_1 \equiv \floor{i_{\ell_2}/q^{m-1}} + d \ell_2 (\bmod q^{x_0})
  \end{align*}
  implies
  \begin{align*}
    \floor{i_{\ell_1}/q^{m-1}} + d \ell_1 \equiv \floor{i_{\ell_2}/q^{m-1}} + d \ell_2 (\bmod q^{x_0+4m-2})
  \end{align*}
\end{remark}

\begin{proof}
  We have already seen that $\{M_{x,c}: c<q^{x}\}$ is a partition of $\{0,\ldots,k-1\}$.
  Furthermore, we find for $0\leq x\leq(4m-2)k$ and $c<q^{x}$ that
  \begin{align*}
    M_{x,c} = \bigcup_{c'<q^{4m-2}} M_{x+(4m-2),c+q^{x}c'}.
  \end{align*}
  This implies that $\{M_{x+4m-2,c}: c<q^{x+4m-2}\}$ is a refinement of $\{M_{x,c}: c<q^{x}\}$ and we find
  \begin{align*}
    &\{M_{(4m-2)\cdot 0,c}: c<1\} \geq \{M_{(4m-2)\cdot 1,c}: c<q^{4m-2}\}\\
		&\qquad \geq \ldots \geq \{M_{(4m-2)k,c}: c<q^{(4m-2)k}\}.
  \end{align*}
  It is well known that the maximal length of a chain in the set of partitions of $\{0,\ldots,k-1\}$ is $k$.
  This means that there exists $x'_0$ such that $\{M_{(4m-2)x'_0,c}: c<q^{(4m-2)x'_0}\} = \{M_{(4m-2)(x'_0+1),c'}: c'<q^{(4m-2)(x'_0+1)}\}$.
  \end{proof}
Furthermore, we define
\begin{align*}
  \beta_{x,c} := \sum_{\ell \in M_{x,c}} \alpha_{\ell}.
\end{align*}

We can now choose $m_1 := (4m-2)k$, $m'_1 := x_0+(4m-2)$ where $x_0$ is given by Lemma~\ref{le:x_0}.
We consider $c_0<q^{x_0}$ and $c^{+}_0$ provided by Lemma~\ref{le:c_0} and Lemma~\ref{le:x_0}
and know that $\beta_{x,c_0} \notin \Z$.
Therefore we apply Corollary~\ref{co:diff_b} and find $\bfe_1,\bfe_2<q^{2m-1}$ such that
\begin{align*}
  &b(q^{m-1}(\bfe_1 + 1)-1) - b(q^{m-1}(\bfe_1+1))\\
	&\qquad - b(q^{m-1}(\bfe_2 + 1)-1) + b(q^{m-1}(\bfe_2+1)) = d
\end{align*}
and $d\beta_{x,c_0}\notin \Z$.

We are now able to define
\begin{align*}
  \varepsilon_1 = (q^{x_0+m-1}(\bfe_1+1)-c_0^{+}-1)\bmod q^{x_0+4m-2}\\
  \varepsilon_2 = (q^{x_0+m-1}(\bfe_2+1)-c_0^{+}-1)\bmod q^{x_0+4m-2}.
\end{align*}

It just remains to check \eqref{eq:goal_eps_12} which we split up into the following two lemmata.
\begin{lemma}
  Let $x_0,\varepsilon_i$ be defined as above.
  Then
  \begin{align*}
    T_{\varepsilon_i,d}^{x_0+4m-2}(I) = T_{\varepsilon_i+1,d}^{x_0+4m-2}(I)
  \end{align*}
  holds.
\end{lemma}

\begin{proof}
We need to show that
\begin{align}\label{eq:equal_TI}
  \floor{\frac{i_{\ell} + q^{m-1}(\ell d + \varepsilon_i)}{q^{x_0+4m-2}}} = \floor{\frac{i_{\ell} + q^{m-1}(\ell d + \varepsilon_i +1)}{q^{x_0+4m-2}}}
\end{align}
holds for all $\ell<k$ and $i = 1,2$. We know that $\ell$ belongs to $M_{x_0+4m-2,c^{+}}$ for some $c<q^{x_0}$.
Thus, we find for $j=0,1$
\begin{align*}
  \floor{\frac{i_{\ell} + q^{m-1}(\ell d + \varepsilon_i +j)}{q^{x_0+4m-2}}} &= \floor{\frac{(i_{\ell} \bmod q^{m-1}) + q^{m-1}(c^{+} + \varepsilon_i +j)}{q^{x_0+4m-2}}}\\
      &= \floor{\frac{c^{+} + \varepsilon_i+j}{q^{x_0+3m-1}}}
\end{align*}

Therefore, \eqref{eq:equal_TI} does hold, unless
\begin{align*}
  c^{+} + \varepsilon_i +1 \equiv 0 (\bmod q^{x_0+3m-1}).
\end{align*}
We find
\begin{align*}
  c^{+} + \varepsilon_i +1 \equiv c^{+} + q^{x_0+m-1}(\bfe_i+1)-c_0^{+} (\bmod q^{x_0+3m-1}).
\end{align*}
We first consider the case $c\neq c_0$: 
\begin{align*}
  c^{+} + \varepsilon_i +1 \equiv c - c_0 \not \equiv 0 (\bmod q^{x_0})
\end{align*}
For $c = c_0$:
\begin{align*}
    c_0^{+} + \varepsilon_i +1 \equiv q^{x_0+m-1}(\bfe_i+1) (\bmod q^{x_0+3m-1})
\end{align*}
However
\begin{align*}
  \bfe_i+1 \not \equiv 0 (\bmod q^{2m})
\end{align*}
as $\bfe_i <q^{2m-1}$.
Thus, \eqref{eq:equal_TI} holds.
\end{proof}

\begin{lemma}
    There exists $\eta'>0$ only depending on $m'$ such that for $x_0$ and $\varepsilon_i$ defined as above holds
    \begin{align}\label{eq:eps_saving}
      \sum_{i=1}^{2}\abs{v^{x_0+4m-2}(I,\varepsilon_i,\delta) + z \cdot v^{x_0+4m-2}(I,\varepsilon_i+1,\delta)} \leq 4-\eta'
    \end{align}
    for all $z\in \U$.
\end{lemma}

\begin{proof}
We start by computing the weights $v^{x_0+4m-2}(I,\varepsilon_i+j,\delta)$.
For arbitrary $\varepsilon < q^{\lambda_0+4m-2}$, we find:
\begin{align*}
  v^{x_0+4m-2}&(I,\varepsilon,d) = \prod_{\ell<k} \e(\alpha_{\ell} b_{x_0+4m-2}(i_{\ell} + q^{m-1}(\varepsilon + \ell d)))\\
    &= \prod_{\ell<k} \e(\alpha_{\ell} b_{m-1}(i_{\ell} + q^{m-1}(\varepsilon + \ell d))) \e\rb{\alpha_{\ell} b_{x_0+3m-1}\rb{\floor{i_{\ell}/q^{m-1}} + \varepsilon + \ell d}}\\
    &= \e(g(\varepsilon)) \cdot \prod_{\ell<k} \e\rb{\alpha_{\ell} b_{x_0+3m-1}\rb{\floor{i_{\ell}/q^{m-1}} + \varepsilon + \ell d}}.
\end{align*}
where
\begin{align*}
  g(\varepsilon) := \sum_{\ell<k} \alpha_{\ell} b_{m-1}(i_{\ell} + q^{m-1}(\varepsilon + \ell d)).
\end{align*}
Note that $g(\varepsilon)$ only depends on $\varepsilon \bmod q^{m-1}$.

We can describe this product by using the weights $\beta$ defined above.
\begin{align*}
  v^{x_0+4m-2}(I,\varepsilon,d) &= \e(g(\varepsilon)) \cdot \prod_{c'<q^{x_0+4m-2}} \e\rb{\beta_{x_0+4m-2,c'} \cdot b_{x_0+3m-1} \rb{c' + \varepsilon}}.
\end{align*}

Furthermore, we can rewrite every $c'<q^{x_0+4m-2}$ for which $\beta_{x_0+4m-2,c'} \not = 0$ as some $c^{+}$ where $c<q^{x_0}$.
This gives then
\begin{align*}
   v^{x_0+4m-2}&(I,\varepsilon,d) = \e(g(\varepsilon)) \cdot \prod_{c<q^{x_0}} \e\rb{\beta_{x_0,c} \cdot b_{x_0+3m-1} \rb{c^{+} + \varepsilon}}\\
    &= \e(g(\varepsilon)) \cdot \prod_{c<q^{x_0}} \e\rb{\beta_{x_0,c} \cdot b_{x_0}(c^{+} + \varepsilon)} \cdot \prod_{c<q^{x_0}} \e\rb{\beta_{x_0,c} \cdot b_{3m-1}\rb{\floor{\frac{c^{+}+\varepsilon}{q^{x_0}}}}}
\end{align*}

Thus we find for $\varepsilon = \varepsilon_i+j$ that:
\begin{align*}
  v^{x_0+4m-2}&(I,\varepsilon_i+j,d) = \e(g(\varepsilon_i+j)) \cdot \prod_{c<q^{x_0}} \e\rb{\beta_{x_0,c} \cdot b_{x_0}(c^{+} + \varepsilon_i+j)}\\
	  &\qquad \cdot \prod_{c<q^{x_0}} \e\rb{\beta_{x_0,c} \cdot b_{3m-1}\rb{\floor{\frac{c^{+}+\varepsilon_i+j}{q^{x_0}}}}}\\
      &= \e(g(-c_0^{+}-1+j)) \cdot \prod_{c<q^{x_0}} \e\rb{\beta_{x_0,c} \cdot b_{x_0}(c^{+} -c_0^{+}-1+j)} \\
	  &\qquad \cdot \prod_{c<q^{x_0}} \e\rb{\beta_{x_0,c} \cdot b_{3m-1}\rb{q^{m-1}(\bfe_i+1)+\floor{\frac{c^{+}-c_0^{+}-1+j}{q^{x_0}}}}}\\
      &= \e(g(-c_0^{+}-1+j)) \cdot \prod_{c<q^{x_0}} \e\rb{\beta_{x_0,c} \cdot b_{x_0}(c^{+} -c_0^{+}-1+j)} \\
	  &\qquad \cdot \prod_{\substack{c<q^{x_0}\\c\neq c_0}} \e\rb{\beta_{x_0,c} \cdot b_{3m-1}\rb{q^{m-1}(\bfe_i+1)+\floor{\frac{c^{+}-c_0^{+}-1+j}{q^{x_0}}}}}\\
	  &\qquad \cdot \e\rb{\beta_{x_0,c_0} \cdot b_{3m-1}(q^{m-1}(\bfe_i+1)-1+j)}.
\end{align*}

For $c \neq c_0$, we find 
\begin{align*}
  \floor{\frac{c^{+}-c_0^{+}-1}{q^{x_0}}} = \floor{\frac{c^{+}-c_0^{+}}{q^{x_0}}}
\end{align*}
as $c^{+} \equiv c \not \equiv c_0 \equiv c_0^{+} \mod q^{x_0}$.

Consequently, we find
\begin{align*}
  &v^{x_0+4m-2}(I,\varepsilon_i,d) = \e(x_i)\\
  &v^{x_0+4m-2}(I,\varepsilon_i+1,d) = \e(x_i+\xi_i)
\end{align*}
where
\begin{align*}
  x_i &= g(-c_0^{+}-1) + \sum_{c<q^{x_0}} \beta_{x_0,c} \cdot b_{x_0}(c^{+} -c_0^{+}-1) \\
	  &\qquad + \sum_{\substack{c<q^{x_0}\\c\neq c_0}} \beta_{x_0,c} \cdot b_{3m-1}\rb{q^{m-1}(\bfe_i+1)+\floor{\frac{c^{+}-c_0^{+}}{q^{x_0}}}}\\
	  &\qquad + \beta_{x_0,c_0} \cdot b_{3m-1}(q^{m-1}(\bfe_i+1)-1)
\end{align*}
and
\begin{align*}
  \xi_i &= g(-c_0^{+})  + \sum_{c<q^{x_0}} \beta_{x_0,c} \cdot b_{x_0}(c^{+} -c_0^{+})  + \beta_{x_0,c_0} \cdot b_{3m-1}(q^{m-1}(\bfe_i+1))\\
      & - g(-c_0^{+}-1)- \sum_{c<q^{x_0}} \beta_{x_0,c} \cdot b_{x_0}(c^{+} -c_0^{+}-1) - \beta_{x_0,c_0} \cdot b_{3m-1}(q^{m-1}(\bfe_i+1)-1).
\end{align*}
Also, we find
\begin{align*}
  \xi_1-\xi_2 = \beta_{x_0,c_0}d \notin \Z,
\end{align*}
where
\begin{align*}
  &b(q^{m-1}(\bfe_1+1)) - b(q^{m-1}(\bfe_1+1)-1)\\
	&\qquad -b(q^{m-1}(\bfe_2+1)) + b(q^{m-1}(\bfe_2+1)-1)=d.
\end{align*}

This implies
\begin{align*}
  \norm{\xi_1-\xi_2} \geq \frac{1}{m'}.
\end{align*}

It remains to apply Lemma~\ref{le:sum_4} to find that \eqref{eq:eps_saving} holds with $\eta' = 8 \rb{\sin\rb{\frac{\pi}{4m'}}}^2$.
\end{proof}

At the end of this section, we recall the important steps of the proof of Proposition~\ref{Pro2_new}.\\
We started to rewrite our recursion for $G_{\lambda}^{I}$ into a matrix vector multiplication
\begin{align*}
    G_{\lambda}(h,q^Ld+\delta) = \frac{1}{q^L} M^{L}_{\delta}\rb{\e\rb{-\frac{h}{q^{\lambda}}}}G_{\lambda-L}\rb{h,d}.
\end{align*}
We then split up this matrix $M^{L}_{\delta}(.)$ into a product of many matrices $M^{m_1}_{\delta_j}(.)$, where $m_1 = (4m-2)k$.
Thereafter, we showed that $\norm{M^{m_1}_{\delta_j}(.)}\leq q^{m_1}-\eta$, where $\eta = 8 \rb{\sin\rb{\frac{\pi}{4m'}}}^2$.
This implies then Proposition~\ref{Pro2_new}.\\
To show that $\norm{M^{m_1}_{\delta_j}}\leq q^{m_1}-\eta$, we found two different $\varepsilon_i$ such that
\begin{align*}
  &T_{\varepsilon_i,\delta}^{m_1'}(I) = T_{\varepsilon_i+1,\delta}^{m_1'}(I),\\
  &\abs{v^{m_1'}(I,\varepsilon_1,\delta) + z v^{m_1'}(I,\varepsilon_1+1,\delta)} + \abs{v^{m_1'}(I,\varepsilon_2,\delta) + z v^{m_1'}(I,\varepsilon_2+1,\delta)} \leq 4-\eta'
\end{align*}
holds for all $z\in \U$.

\section{Proof of the Main Theorem} \label{cha:proof}
In this section, we complete the proof of Theorem \ref{Thexponentialsums} following the ideas and structure of \cite{drmotaMauduitRivat2014}. 
As the proof is very similar, we only outline it briefly and comment on the important changes.

The structure of the proof is similar for both cases:
At first we want to substitute the function $b$ by $b_{\mu,\lambda}$. This can be done by applying Lemma 
\ref{Lecarry0} and Lemma \ref{lemma:van-der-corput} in the case $K \in \Z$. 
For the case $K \notin \Z$ we have to use Lemma \ref{lemma:van-der-corput} first.

Thereafter, we apply Lemma \ref{Lecarry1} to detect the digits between $\mu$ and $\lambda$. 
Next, we use characteristic functions to detect suitable values for $u_1(n), u_2(n), u_3(n)$. 
Lemma \ref{lemma:better-koksma1} allows us to replace the characteristic functions by exponential sums. 
We split the remaining exponential sum into a quadratic and a linear part and find that the quadratic part is negligibly small. 
For the remaining sum, we apply Proposition \ref{Pro1} or \ref{Pro2}~-- depending on whether $K \in \Z$. 
The case $K \notin \Z$ needs more effort to deal with.
\subsection{\texorpdfstring{The case $K \in \Z$}{The case K in Z}} \label{sec:equiv0}

In this section, we show that, if $K = \alpha_0 + \cdots +
\alpha_{k-1} \in \Z$, 
Proposition \ref{Pro1} 
provides an upper bound for the sum
\begin{displaymath}
S_0 = \sum_{n < N} \e\rb{ \sum_{\ell=0}^{k-1}  \alpha_\ell b((n+\ell)^2) }.
\end{displaymath}

Let $\nu$ be the unique integer such that $q^{\nu-1} < N \leq q^{\nu}$
and we choose all appearing exponents - i.e. $\lambda,\mu,\rho$, etc. - as in \cite{drmotaMauduitRivat2014}.

By using Lemma~\ref{Lecarry0}, and the same arguments as in \cite{drmotaMauduitRivat2014}, we find
\begin{align}\label{eq:S0-S1-even}
  S_0= S_1 + \mathcal{O}\rb{q^{\nu - (\lambda-\nu)}},
\end{align}
where
\begin{displaymath}
S_1 = \sum_{n< N} \e\rb{ \sum_{\ell=0}^{k-1}  \alpha_\ell 
b_\lambda((n+\ell)^2) }.
\end{displaymath} 

Now we use Lemma \ref{lemma:van-der-corput} - with $Q= q^{\mu+m-1}$ and $S = q^{\nu-\mu}$ - to relate $S_1$ to a sum in terms of $b_{\mu,\lambda}$:

\begin{align}\label{eq:S1-S2-even}
  |S_1|^2 \ll \frac {N^2}S + \frac NS \Re(S_2),
\end{align}
where
\begin{align*}
S_2 = \sum_{1\le s < S} \rb{ 1 - \frac sS } S_2'(s)
\end{align*}
and
\begin{align*}
S_2'(s) = 
\sum_{n\in I(N,s)} \e\rb{ \sum_{\ell=0}^{k-1} \alpha_\ell 
(b_{\mu,\lambda}((n+\ell)^2)
- b_{\mu,\lambda}((n+\ell+sq^{\mu+m-1})^2))},
\end{align*}
where 
$I(N,s)$ is an interval included in $[0,N-1]$ (which we do not specify).

Next we use Lemma \ref{Lecarry1} to detect the digits of $(n+\ell)^2$ and $(n+\ell + sq^{m-1}q^{\mu})^2$ between $\mu$ and $\lambda+m-1$ - with a negligible error term.
Therefore, we have to take the digits between 
$\mu' = \mu-\rho'$ and $\mu$ into account, where $\rho'>0$ will be chosen later.

We set the integers $u_1=u_1(n)$, $u_3=u_3(n)$, $v=v(n)$, $w_1=w_1(n)$,
and $w_3=w_3(n)$ to satisfy the conditions of Lemma~\ref{Lecarry1} and detect them by characteristic functions.
Thus, we find
\begin{align}\label{eq:S'2-S'3-even}
  S_2'(s) = S_3'(s) + \mathcal{O}(q^{\nu-\rho'}),
\end{align}
where
\begin{small}
\begin{align*}
  S_3'(s) &= \sum_{0\le u_1 < U_1} \sum_{0\le u_3 < U_3} \sum_{n\in I(N,s)} \left(\chi_{q^{\mu'-\lambda-m+1}}
      \rb{\frac{n^2}{q^{\lambda+m-1}}-\frac{u_1}{U_1}} \chi_{q^{\mu'-\nu-1}} \rb{\frac{2n}{q^{\nu+1}}-\frac{u_3}{U_3}} \right.\\
    &\left. \cdot \e\rb{ \sum_{\ell=0}^{k-1} \alpha_\ell (b_{\rho',\lambda-\mu+\rho'}(u_1+\ell u_3)
      - b_{\rho',\lambda-\mu+\rho'}(u_1+\ell u_3+ v(n) q^{\rho'} + 2 \ell s q^{m-1}q^{\rho'})}\right),
\end{align*}
\end{small}
where $\chi_{\alpha}$ is defined by \eqref{eq:definition-chi} and $U_1 = q^{\lambda+m-1-\mu'}, U_3 = q^{\nu-\mu'+1}$.
Lemma \ref{lemma:better-koksma1} allows us to replace the 
characteristic functions $\chi$  by trigonometric polynomials.
More precisely, using \eqref{eqS-S}
with
$H_1 = U_1 q^{\rho''}$ and $H_3 = U_3 q^{\rho''}$
for some suitable $\rho'' > 0$ (which is a fraction of $\nu$ chosen later),
we have
\begin{align}\label{eq:S'3-S4}
  S'_3(s) = S_4(s)  + \mathcal{O}(E_1) + \mathcal{O}(E_3) + \mathcal{O}(E_{1,3}),
\end{align}
where $E_1, E_3$ and $E_{1,3}$ are the error terms specified in \eqref{eqS-S} and 
\begin{align*}
  S_4(s&) = \sum_{0\le u_1 < U_1} 
\sum_{0\le u_3 < U_3} \sum_{0\le v < q^{\lambda-\mu+m-1} } \\
  &\sum_{n\in I(N,s)} \left ( A_{U_1^{-1},H_1} \rb{\frac{n^2}{q^{\lambda+m-1}}-\frac{u_1}{U_1}} A_{U_3^{-1},H_3} \rb{\frac{2n}{q^{\nu+1}}-\frac{u_3}{U_3}} \right . \\
    &\cdot \e\rb{ \sum_{\ell=0}^{k-1} \alpha_\ell (b_{\rho',\lambda-\mu+\rho'}(u_1+\ell u_3)
      - b_{\rho',\lambda-\mu+\rho'}(u_1+\ell u_3+ v q^{\rho'} + 2 \ell s q^{m-1} q^{\rho'}))}\\
  &\cdot \left . \frac 1{q^{\lambda-\mu+m-1}} \sum_{0\le h < q^{\lambda-\mu+m-1}} \e\rb{ h\frac{2sq^{m-1}n-v}{q^{\lambda-\mu+m-1}}}\right),
\end{align*}
where we use the last sum to detect the correct value of $v = v(n)$.
 
The error terms $E_1$, $E_3$, $E_{1,3}$ can easily be 
estimated with the help of Lemma~\ref{lemma:incomplete-gauss-sum}, just as in \cite{drmotaMauduitRivat2014}.

By using the representations of $A_{U_1^{-1},H_1}$ and
$A_{U_3^{-1},H_3}$, we obtain
\begin{align*}
  S_4(s&) = \frac 1{q^{\lambda-\mu+m-1}} \sum_{|h_1| \le H_1} \sum_{|h_3| \le H_3} \sum_{0\le h < q^{\lambda-\mu+m-1} } a_{h_1}(U_1^{-1},H_1)\,a_{h_3}(U_3^{-1},H_3) \\
    & \sum_{0\le u_1 < U_1} \sum_{0\le u_3 < U_3} \sum_{0\le v < q^{\lambda-\mu+m-1}} \e\Biggl(- \frac{h_1u_1}{U_1} - \frac{h_3u_3}{U_3} - \frac{hv}{q^{\lambda-\mu+m-1}} \Biggr) \\
    & \e\Biggl( \sum_{\ell=0}^{k-1} \alpha_\ell (b_{\rho',\lambda-\mu+\rho'}(u_1 + \ell u_3)
      - b_{\rho',\lambda-\mu+\rho'}(u_1 + \ell u_3+ v q^{\rho'} +2 \ell s q^{m-1} q^{\rho'})) \Biggr) \\
    & \cdot \sum_{n} \e\rb{ \frac{h_1n^2}{q^{\lambda+m-1}} +\frac{h_3n}{q^\nu} + \frac{2hsn}{q^{\lambda-\mu}}}.
\end{align*}
We now distinguish the cases $h_1 = 0$ and $h_1\ne 0$.
For $h_1\ne 0$, we can estimate the exponential sum by using Lemma~\ref{lemma:incomplete-gauss-sum} and the following estimate
\begin{align}
	\sum_{1\leq h_1\leq H_1}\sqrt{\gcd(h_1,q^{\lambda})} \ll_{q} H_1.
\end{align}
Thus, we find
\begin{align*}
  \sum_{0<|h_1| \le H_1} \sum_{|h_3| \le H_3} \sum_{h=0}^{q^{\lambda-\mu+m-1} -1} \left|\sum_n \e\rb{ \frac{h_1n^2}{q^{\lambda+m-1}} +\frac{h_3n}{q^\nu} + 
      \frac{2hsn}{q^{\lambda-\mu}}} \right| \ll \lambda H_1 H_3 q^{\lambda/2  + \lambda-\mu}.
\end{align*}
This gives then
\begin{align}\label{eq:S4-S5-even}
S_4(s) = S_5(s) + \mathcal{O}(\lambda q^{3\lambda/4}),
\end{align}
where $S_5(s)$ denotes the part of $S_4(s)$ with $h_1 = 0$.

We set $u_1 = u_1'' + q^{\rho'} u_1'$ and $u_3 = u_3'' + q^{\rho'} u_3'$ (where $0\le u_1'', u_3'' < q^{\rho'}$). 
Furthermore, we define $i_\ell = \lfloor (u_1''+\ell u_3'')/q^{\rho'}\rfloor$.
As $I = (i_\ell)_{0\le \ell < k} = (\lfloor (u_1''+\ell u_3'')/q^{\rho'}\rfloor)_{0\le \ell < k}$ is contained in 
$\mathcal{I}'_k$, we have - by the same arguments as in \cite{drmotaMauduitRivat2014} - 
\begin{align*}
  S_5(s) &\le \sum_{|h_3| \le H_3} \sum_{0\le h < q^{\lambda-\mu+m-1}} \frac 1{q^{\nu+1-\mu}} \sum_{0\le u_3' < q^{\nu-\mu+1}}\\
      &\qquad \sum_{I \in \mathcal{I}_k} \left| H_{\lambda-\mu}^{I}(h,u_3') \overline {H_{\lambda-\mu}^{I}(h,u_3'+2sq^{m-1})}  \right|\\
      &\qquad \cdot \min\rb{ N, \left| \sin\rb{ \pi \rb{ \frac{h_3}{q^\nu} + \frac{2hs}{q^{\lambda-\mu}}}} \right|^{-1} }.
\end{align*}
Using the estimate $\left|H_{\lambda-\mu}^{I}(h,u_3'+2sq^{m-1})\right|\le 1$ and 
the Cauchy-Schwarz inequality, yields
\begin{align*}
  \sum_{0\le u_3' < q^{\nu-\mu+1}} &\left| H_{\lambda-\mu}^{I}(h,u_3') \overline {H_{\lambda-\mu}^{I}(h,u_3'+2sq^{m-1})}  \right|\\
    &\le q^{(\nu-\mu+1)/2} \rb{ \sum_{0\le u_3' < q^{\nu-\mu+1}} \left| H_{\lambda-\mu}^{I}(h,u_3')  \right| ^2 }^{1/2}.
\end{align*}
We now replace $\lambda$ by $\lambda-\mu+m-1$, $\lambda'$ by $\nu-\mu+1$ and apply Proposition~\ref{Pro1}.
\begin{align*}
  S_5(s) &\ll  q^{-\eta(\lambda-\mu)/2} \sum_{|h_3| \le H_3} \sum_{h = 0}^{q^{\lambda-\mu+m-1} -1} 
      \min\rb{ N, \left| \sin\rb{ \pi \rb{ \frac{h_3}{q^\nu} + \frac{2hs}{q^{\lambda-\mu+m-1}}}} \right|^{-1} }.
\end{align*}

Next we average over $s$ and $h$, as in \cite{drmotaMauduitRivat2014}, by applying Lemma~\ref{le:sum_sum_sin}.
Thus we have a factor $\tau(q^{\lambda-\mu}) \ll_{q} (\lambda-\mu)^{\omega(q)}$ compared to $\tau(2^{\lambda-\mu}) = \lambda-\mu+1$.
Combining all the estimates as in \cite{drmotaMauduitRivat2014} gives then
\begin{align*}
  |S_0| \ll q^{\nu-(\lambda-\nu)} + \nu^{(\omega(q)+1)/2} q^\nu q^{-\eta (\lambda-\nu)/2} 
      + q^{\nu-\rho'/2} + q^{\nu - \rho''/2} + \lambda^{1/2} q^{\nu/2 + 3\lambda/8}
\end{align*}
~-- provided that the following conditions hold
\begin{align*}
&2\rho' \le \mu \le \nu-\rho', 
\quad  \rho'' < \mu'/2, \quad
\mu' \ll 2^{\nu- \mu'}, \quad
2\mu' \ge \lambda, \\
&(\nu-\mu) + 2(\lambda-\mu) + 2(\rho'+\rho'') \le \lambda/4, \quad
\nu-\mu'+\rho'' + \lambda - \mu \le \nu.
\end{align*}
For example, the choice
\begin{align*}
\lambda = \nu+\left\lfloor \frac{\nu}{20}\right\rfloor \mbox{ and } \rho' = \rho'' = \left\lfloor \frac{\nu}{200}\right\rfloor
\end{align*}
ensures that the above conditions are satisfied.

Summing up we proved that for $\eta'<\min(1/200,\eta/40)$ - where $\eta$ is given by Proposition~\ref{Pro1} - holds
\begin{align*}
S_0 \ll q^{\nu(1-\eta')} \ll N^{1-\eta'}
\end{align*}
which is precisely the statement of Theorem~\ref{Thexponentialsums}.

\subsection{\texorpdfstring{The case $K \not \in \Z$}{The case K not in Z}}
\label{sec:nequiv0}

In this section, we show that, for $K = \alpha_0 + \cdots +
\alpha_{k-1} \not \in \Z$, 
Proposition \ref{Pro2} 
provides an upper bound for the sum
\begin{align*}
  S_0 = \sum_{n< N} \e\rb{\sum_{\ell=0}^{k-1} \alpha_\ell b((n+\ell)^2)}.
\end{align*}

Let $\mu$, $\lambda$, $\rho$ and $\rho_1$ be integers satisfying
\begin{align}\label{eq:mu-lambda-rho}
  0 \leq \rho_1 < \rho < \mu = \nu-2\rho < \nu  < \lambda = \nu+2\rho < 2\nu
\end{align}
to be chosen later - just as in \cite{drmotaMauduitRivat2014}.
Since $K \not \in \Z$ we can not use Lemma \ref{Lecarry0} directly.
Therefore, we apply Lemma~\ref{lemma:van-der-corput} with $Q=1$
and $R=q^\rho$. Summing trivially for $1\leq r \leq R_1=q^{\rho_1}$ yields
\begin{align*}
  \abs{S_0}^2 \ll \frac{N^2R_1}{R} + \frac{N}{R} \sum_{R_1 < r < R} \rb{1-\frac{r}{R}} \Re(S_1(r)),
\end{align*}
where
\begin{align*}
  S_1(r) = \sum_{n \in I_1(r)} \e\rb{\sum_{\ell=0}^{k-1} \alpha_\ell \rb{ b((n+\ell)^2) - b((n+r+\ell)^2)}}
\end{align*}
and $I_1(r)$ is an interval included in $[0,N-1]$. 
By Lemma~\ref{Lecarry0} we conclude that $b_{\lambda,\infty}((n+\ell)^2) = b_{\lambda,\infty}((n+r+\ell)^2)$ for all but $\mathcal{O}(Nq^{-(\lambda-\nu-\rho)})$ values of $n$. 
Therefore, we see that
\begin{align*}
    S_1(r) = S_1'(r) + \mathcal{O}(q^{\nu-(\lambda-\nu-\rho)}),
\end{align*}
with
\begin{align*}
  S_1'(r) = \sum_{n \in I_1(r)} \e\rb{\sum_{\ell=0}^{k-1} \alpha_\ell \rb{ b_\lambda((n+\ell)^2) - b_\lambda((n+r+\ell)^2)}}.
\end{align*}
This leads to
\begin{align*}
  \abs{S_0}^2 \ll q^{2\nu-\rho+\rho_1} + q^{3\nu+\rho-\lambda} + \frac{q^{\nu}}{R} \sum_{R_1 < r < R} \abs{S'_1(r)}
\end{align*}
and, by using the Cauchy-Schwarz inequality to
\begin{align*}
  \abs{S_0}^4 \ll q^{4\nu-2\rho+2\rho_1} + q^{6\nu+2\rho-2\lambda} + \frac{q^{2\nu}}{R} \sum_{R_1 < r < R} \abs{S'_1(r)}^2.
\end{align*}
For $\abs{S'_1(r)}^2$ we can use Lemma~\ref{lemma:van-der-corput} again:
Let $\rho'\in\N$ to be chosen later such that 
$1\leq \rho'\leq \rho$.
After applying Lemma~\ref{lemma:van-der-corput}
with $Q=q^{\mu+m-1}$ and
\begin{align} \label{eq:definition-S}
  S = q^{2\rho'} \leq q^{\nu-\mu},
\end{align}
we observe that for any $\widetilde{n}\in\N$ we have
\begin{align*}
  b_{\lambda}((\widetilde{n}+sq^{\mu+m-1})^2) - b_{\lambda}(\widetilde{n}^2) 
    = b_{\mu,\lambda}((\widetilde{n}+sq^{\mu+m-1})^2) - b_{\mu,\lambda}(\widetilde{n}^2),
\end{align*}
and thus
\begin{align} \label{eq:S0-S2}
  \abs{S_0}^4 \ll q^{4\nu-2\rho+2\rho_1} + q^{6\nu+2\rho-2\lambda} + \frac{q^{4\nu}}{S} + \frac{q^{3\nu}}{RS} \sum_{R_1 < r < R} \sum_{1\leq s < S} \abs{S_2(r,s)},
\end{align}
with
\begin{align*}
  S_2(r,s) &= \sum_{n \in I_2(r,s)} \e\Biggl(\sum_{\ell=0}^{k-1} \alpha_\ell \bigl(b_{\mu,\lambda}((n+\ell)^2) - b_{\mu,\lambda}((n+r+\ell)^2) \\
    &\qquad - b_{\mu,\lambda}((n+sq^{\mu+m-1}+\ell)^2) + b_{\mu,\lambda}((n+sq^{\mu+m-1}+r+\ell)^2)\bigr) \Biggr),
\end{align*}
where $I_2(r,s)$ is an interval included in $[0,N-1]$.

We now make a Fourier analysis similar to the case $K \equiv 0 (\bmod 1)$ - as in \cite{drmotaMauduitRivat2014}.
We set $U = q^{\lambda+m-1-\mu'}, U_3 = q^{\nu-\mu'+1}$ and $V = q^{\lambda-\mu+m-1}$.
We apply Lemma~\ref{Lecarry1} and detect the correct values of $u_1,u_2,u_3$ by characteristic functions.
This gives

\begin{align*}
  S_2(r,s) &= \sum_{0\leq u_1 < U} \sum_{0\leq u_2 < U} \sum_{0\leq u_3 < U_3}\\
      & \sum_{n \in I_2(r,s)} \e\Biggl(\sum_{\ell=0}^{k-1} \alpha_\ell \bigl(b_{\rho',\lambda-\mu+\rho'}(u_1+\ell u_3) - b_{\rho',\lambda-\mu+\rho'}(u_2+\ell u_3) \\
      & \qquad \qquad - b_{\rho',\lambda-\mu+\rho'}(u_1+\ell u_3+ v(n) q^{\rho'} + 2 \ell s q^{m-1} q^{\rho'})\\
      & \qquad \qquad + b_{\rho',\lambda-\mu+\rho'}(u_2+\ell u_3+ v(n) q^{\rho'} + 2(\ell+r) s q^{m-1} q^{\rho'}) \bigr) \Biggr) \\
      & \chi_{U^{-1}}\rb{\frac{n^2}{q^{\lambda+m-1}} - \frac{u_1}{U}} \chi_{U^{-1}}\rb{\frac{(n+r)^2}{q^{\lambda+m-1}} - \frac{u_2}{U}} 
        \chi_{U_3^{-1}}\rb{\frac{2n}{q^\nu} - \frac{u_3}{U_3}}\\
      & \qquad +\mathcal{O}(q^{\nu-\rho'}).
\end{align*}
Furthermore, we use Lemma \ref{lemma:better-koksma1} to replace the 
characteristic functions $\chi$  by trigonometric polynomials.
Using \eqref{eqS-S} with
$U_1=U_2=U$,
$H_1=H_2 = U q^{\rho_2}$ and 
$H_3 = U_3 q^{\rho_3}$,
and integers $\rho_2$,
$\rho_3$ verifying
\begin{align}\label{eq:condition-rho2-rho3}
  \rho_2 \leq \mu-\rho',\ \rho_3 \leq \mu-\rho',\ 
\end{align}
we obtain
\begin{align*}
  S_2(r,s) = S_3(r,s) &+ \mathcal{O}(q^{\nu-\rho'}) + \mathcal{O}\rb{ E_{30}(r)} + \mathcal{O}\rb{ E_{31}(0)} + \mathcal{O}\rb{ E_{31}(r)}\\
      & + \mathcal{O}\rb{E_{32}(0)} + \mathcal{O}\rb{E_{32}(r)} + \mathcal{O}\rb{E_{33}(r)} + \mathcal{O}\rb{E_{34}(r)},
\end{align*}
for the error terms obtained by~\eqref{eqS-S} and $S_3(r,s)$ obtained by replacing the characteristic function by trigonometric polynomials.
We now reformulate $S_3(r,s)$ by expanding the trigonometric polynomials, detecting the correct value of $v = v(n)$ and restructuring the sums:
\begin{align*}
  S_3(r,s) &= \frac{1}{q^{\lambda-\mu+m-1}} \sum_{0\leq h < q^{\lambda-\mu+m-1}} \sum_{\abs{h_1} \leq H_1} a_{h_1}(U^{-1},H_1)\\
    &\quad \sum_{\abs{h_2} \leq H_2} a_{h_2}(U^{-1},H_2) \sum_{\abs{h_3} \leq H_3} a_{h_3}(U_3^{-1},H_3)\\
    &\quad \sum_{0\leq u_1 < U} \sum_{0\leq u_2 < U} \sum_{0\leq u_3 < U_3} \sum_{0\leq v < V} 
      \e\rb{- \frac{h_1 u_1+h_2 u_2}{U} - \frac{h_3u_3}{U_3} -\frac{hv}{q^{\lambda-\mu+m-1}}}\\
    &\quad \e\Biggl(\sum_{\ell=0}^{k-1} \alpha_\ell \bigl(b_{\rho',\lambda-\mu+\rho'}(u_1+\ell u_3) - b_{\rho',\lambda-\mu+\rho'}(u_2+\ell u_3)\\
    & \quad \qquad \qquad - b_{\rho',\lambda-\mu+\rho'}(u_1+\ell u_3+ vq^{\rho'} + 2\ell s q^{m-1} q^{\rho'})\\
    & \quad \qquad \qquad + b_{\rho',\lambda-\mu+\rho'}(u_2+\ell u_3+ vq^{\rho'} + 2(\ell+r) s q^{m-1} q^{\rho'}) \bigr) \Biggr)\\
    & \quad \sum_{n \in I_2(r,s)} \e\rb{\frac{h_1n^2 + h_2(n+r)^2}{q^{\lambda+m-1}} + \frac{2h_3n}{q^\nu} + \frac{2hsn}{q^{\lambda-\mu}}}.
\end{align*}

One can estimate the error terms just as in \cite{drmotaMauduitRivat2014} and finds that they are bounded by either $q^{\nu-\rho_3}$ or $q^{\nu-\rho_2}$.
In conclusion we deduce that
\begin{align}\label{eq:S2S3_final}
  S_2(r,s) = S_3(r,s) + \mathcal{O}(q^{\nu-\rho'}) + \mathcal{O}(q^{\nu-\rho_2}) + \mathcal{O}(q^{\nu-\rho_3}).
\end{align}

We now split the sum $S_3(r,s)$ into two parts:
\begin{align} \label{eq:S3-S4-S'4}
  S_3(r,s) = S_4(r,s) + S'_4(r,s),
\end{align}
where $S_4(r,s)$ denotes the contribution of the terms for which
$h_1+h_2=0$ while $S'_4(r,s)$ denotes the contribution of the terms
for which $h_1+h_2\neq 0$.
We can estimate $S'_4(r,s)$ just as in \cite{drmotaMauduitRivat2014} and find
\begin{align*}
  S'_4(r,s) &\ll \nu^4 q^{\nu+\frac{1}{2}(8\lambda-9\mu+7\rho'+\rho_2)}
\end{align*}
and it remains to consider $S_4(r,s)$.
Setting $u_1 = u_1'' + q^{\rho'} u'_1$,
$u_2 = u_2'' + q^{\rho'} u'_2$  and $u_3 = u_3'' + q^{\rho'} u'_3$, 
(where $0\le u_1'', u_2'', u_3'' < q^{\rho'}$) 
we can replace the two-fold restricted block-additive function by a truncated block-additive function
\begin{small}
\begin{align*}
  &b_{\rho',\lambda-\mu+\rho'}(u_1 + \ell u_3) = b_{\lambda-\mu}\rb{u_1' + \ell u_3' + \floor{(u_1''+\ell u_3'')/q^{\rho'}}}, \\
  &b_{\rho',\lambda-\mu+\rho'}(u_2 + \ell u_3) = b_{\lambda-\mu}\rb{u_2' + \ell u_3' + \floor{(u_2''+\ell u_3'')/q^{\rho'}}}, \\
  &b_{\rho',\lambda-\mu+\rho'}(u_1 + \ell u_3+ vq^{\rho'} +2 \ell s q^{m-1} q^{\rho'}) \\
      &\quad = b_{\lambda-\mu}\rb{u_1' +v  + \ell (u_3'+2sq^{m-1}) +  \floor{(u_1''+\ell u_3'')/q^{\rho'}}}\\
  &b_{\rho',\lambda-\mu+\rho'}(u_2+\ell u_3+ vq^{\rho'} + 2(\ell+r) s q^{m-1} q^{\rho'}) \\
      &\quad = b_{\lambda-\mu}\rb{u_2' + v + 2sr q^{m-1} + \ell (u_3'+2sq^{m-1}) + \floor{(u_2''+\ell u_3'')/q^{\rho'}}}.
\end{align*}
\end{small}
Using the periodicity of $b$ modulo $V:=q^{\lambda-\mu+ m-1}$, 
we replace the variable $v$ by $v_1$ such that 
$v_1 \equiv u'_1+v (\bmod q^{\lambda-\mu+m-1})$. Furthermore we introduce a new
variable $v_2$ such that 
\begin{align*}
  v_2 \equiv u_2' + v + 2sr q^{m-1} \equiv v_1 + u_2' -u'_1 + 2sr q^{m-1} (\bmod q^{\lambda-\mu+m-1}).
\end{align*}
We then follow the arguments of \cite{drmotaMauduitRivat2014} and find
\begin{align*}
  S_4(r,s) &\ll q^{2\lambda-2\mu} \sum_{h =0}^{q^{\lambda-\mu+m-1}-1} \sum_{h'=0}^{q^{\lambda-\mu+m-1}-1} \sum_{\abs{h_2} \leq H_2} \min(U^{-2},h_2^{-2}) \\
      &\quad \sum_{\abs{h_3} \leq H_3} \min(U_3^{-1},h_3^{-1}) \sum_{0\leq u''_1 < q^{\rho'}} \sum_{0\leq u''_2 < q^{\rho'}} \sum_{0\leq u''_3 < q^{\rho'}} \sum_{0\leq u'_3 < U'_3}\\
      &\qquad \abs{H_{\lambda-\mu}^{I(u''_1,u''_3)}(h'-h-h_2,u_3')} \abs{H_{\lambda-\mu}^{I(u''_2,u''_3)}(h'-h_2,u_3')}\\
      &\quad \qquad \abs{H_{\lambda-\mu}^{I(u''_1,u''_3)}(h'-h,u_3'+2sq^{m-1})} \abs{H_{\lambda-\mu}^{I(u''_2,u''_3)}(h',u_3'+2s q^{m-1})}\\
      & \qquad \qquad \abs{\sum_{n \in I_2(r,s)} \e\rb{\frac{2h_2 rn}{q^{\lambda+m-1}} + \frac{2h_3n}{q^\nu} + \frac{2hsn}{q^{\lambda-\mu}}}},
\end{align*}
with
\begin{align*}
  I(u,\tilde{u}) = \rb{\floor{\frac{u}{q^{\rho'}}}, \floor{\frac{u+\tilde{u}}{q^{\rho'}}}, \ldots, \floor{\frac{u+(k-1)\tilde{u}}{q^{\rho'}}}} 
      \,\text{ for } (u,\tilde{u})\in \N^2. 
\end{align*}

The next few steps are again very similar to the corresponding ones in \cite{drmotaMauduitRivat2014} and we skip the details.
We find
\begin{align*}
  S_4(r,s) &\ll (\lambda-\mu)\ \gcd(2s,q^{\lambda-\mu})\ q^{2\lambda-2\mu} \sum_{0\leq u''_1, u''_2, u''_3 < q^{\rho'}} \sum_{\abs{h_2} \leq H_2} \min(U^{-2},h_2^{-2})\\
    & \qquad \qquad S_6(h_2,s,u''_1,u''_3)^{1/2} S_6(h_2,s,u''_2,u''_3)^{1/2}\\
    & \qquad \sum_{\abs{h_3} \leq H_3} \min(U_3^{-1},h_3^{-1}) \min\rb{q^\nu, \abs{\sin \pi\tfrac{2 h_2 r+2q^{\lambda-\nu+m-1} h_3}{q^{\lambda+m-1}} }^{-1}}.
\end{align*}
where
\begin{align}\label{eq:def-S6}
  S_6(h_2,s,u'',u''_3) &= \sum_{0\leq u'_3 < U'_3} \sum_{0\leq h' < q^{\lambda-\mu+m-1}} \\
    &\qquad \abs{H_{\lambda-\mu}^{I(u'',u''_3)}(h'-h_2,u_3')}^2 \abs{H_{\lambda-\mu}^{I(u'',u''_3)}(h',u_3'+2sq^{m-1})}^2.
\end{align}

Here we 
introduce the integers $H'_2$ and $\kappa$ such that
\begin{align}\label{eq:def-H'2}
  H'_2 = q^{\lambda-\nu+m} H_3/R_1 = q^{\lambda-\mu+\rho'+\rho_3-\rho_1+m+1} = q^\kappa.
\end{align}
This leads to
\begin{align*}
  S_4(r,s) \ll S_{41}(r,s) + S_{42}(r,s) + S_{43}(r,s),
\end{align*}
where $S_{41}(r,s)$, $S_{42}(r,s)$ and $S_{43}(r,s)$ denote the contribution of the terms 
$\abs{h_2}\leq H'_2$, 
$H'_2 < \abs{h_2}\leq q^{\lambda+m-1-\mu}$ and 
$q^{\lambda+m-1-\mu} < \abs{h_2}\leq H_2$ respectively.

\paragraph{\texorpdfstring{Estimate of $S_{41}(r,s)$}{Estimate of S 41}} 
By \eqref{eq:sum_inverse_sinus}
we have
\begin{align*}
  \sum_{\abs{h_3} \leq H_3} \min\rb{q^\nu, \abs{\sin \pi\tfrac{2h_3 + 2h_2 rq^{\nu-\lambda-m+1}}{q^{\nu}} }^{-1} } \ll \nu q^\nu,
\end{align*}
and, therefore,
\begin{align*}
  S_{41}(r,s) &\ll \nu\, (\lambda-\mu)\ \gcd(2s,q^{\lambda-\mu})\ q^{\nu+2\lambda-2\mu} U^{-2} U_3^{-1}\\
    & \qquad \sum_{0\leq u''_1, u''_2, u''_3 < q^{\rho'}} \sum_{\abs{h_2} \leq H'_2} S_6(h_2,s,u''_1,u''_3)^{1/2} S_6(h_2,s,u''_2,u''_3)^{1/2}.
\end{align*}
By Proposition \ref{Pro2}
(replacing $\lambda$ by $\lambda-\mu$ and  $L$ by
$\lambda-\mu-\kappa$),  we find some $0<\eta'\leq 1$ such that
\begin{align*}
\abs{H_{\lambda-\mu}^{I(u'',u''_3)}(h'-h_2,u_3')} \ll q^{-\eta' (\lambda - \mu - \kappa)} 
    \max_{J\in \mathcal{I}_k} \abs{G_{\kappa}^{J}(h'-h_2,\lfloor u_3'/q^L \rfloor)}.
\end{align*}
By Parseval's equality 
and recalling that $\#(\mathcal{I}_k) = q^{m-1} (q^{m-1}+1)^{k-1}$, 
it follows that
\begin{align*}
  & \sum_{\abs{h_2} \leq H'_2} \max_{J\in\mathcal{I}_k} \abs{H_{\kappa}^{J}\lfloor(h'-h_2, u_3'/q^L \rfloor)}^2 \\
  &\qquad \le \sum_{J\in\mathcal{I}_k} \sum_{\abs{h_2} \leq H'_2} \abs{G_{\kappa}^{J}(h'-h_2,\lfloor u_3'/q^L \rfloor)}^2  
      \leq q^{m-1} (q^{m-1}+1)^{k-1}.
\end{align*}
We obtain
\begin{align*}
  \sum_{\abs{h_2} \leq H'_2} \abs{H_{\lambda-\mu}^{I(u'',u''_3)}(h'-h_2,u_3')}^2 
      \ll q^{-\eta' (\lambda - \mu - \kappa)} = \rb{\frac{H'_2}{q^{\lambda-\mu}}}^{\eta'}
\end{align*}
uniformly in $\lambda$, $\mu$, $H'_2$, $u'_3$, $u''$ and $u''_3$.

The remaining proof is analogue to the corresponding proof in \cite{drmotaMauduitRivat2014}.
The only difference is again that by using Lemma~\ref{le:sum_sum_sin} we obtain a factor $(\lambda-\mu)^{\omega(q)}$ instead of $(\lambda-\mu)$.
This gives
\begin{align}
  \label{eq:S41-final}
  \frac{1}{RS} \sum_{R_1 < r < R} \sum_{1\leq s < S} S_{41}(r,s) \ll \nu\, (\lambda-\mu)^{\omega(q) +1}\ q^{\nu - \eta'(\rho_1-\rho'-\rho_3)},
\end{align}
which concludes this part.

\paragraph{\texorpdfstring{Estimate of $S_{42}(r,s)$ and $S_{43}(r,s)$}{Estimate of S 42 and S 43}} 

By following the arguments of \cite{drmotaMauduitRivat2014} and applying the same changes as in the estimate of $S_{41}$ we find
\begin{align}
  \label{eq:S42-final}
  \frac{1}{RS} \sum_{R_1 < r < R} \sum_{1\leq s < S} S_{42}(r,s) \ll \rho (\lambda-\mu)^{2+\omega(q)}\ q^{\nu-\rho+\rho_1+\rho'-\rho_3}
\end{align}
and
\begin{align}\label{eq:S43-final}
  \frac{1}{RS} \sum_{R_1 < r < R} \sum_{1\leq s < S} S_{43}(r,s) \ll \rho\ (\lambda-\mu)^{2+\omega(q)}\ q^{\nu-\rho+3\rho'}.
\end{align}

\paragraph{Combining the estimates for $S_4$}
It follows from \eqref{eq:S41-final}, \eqref{eq:S42-final} and
\eqref{eq:S43-final} that
\begin{align*}
  \frac{1}{RS} \sum_{R_1 < r < R} \sum_{1\leq s < S} S_4(r,s) \ll \nu^{3+\omega(q)} q^{\nu} \rb{q^{ - 2\eta'(\rho_1-\rho'-\rho_3)} + q^{-\rho_3} + q^{-\rho+3\rho'}}.
\end{align*}
Choosing
\begin{align*}
  \rho_1 = \rho-\rho',\
  \rho_2 = \rho_3 = \rho',
\end{align*}
we obtain
\begin{align*}
  \frac{1}{RS} \sum_{R_1 < r < R} \sum_{1\leq s < S} S_4(r,s) \ll \nu^{3+\omega(q)} q^{\nu} \rb{q^{ - 2 \eta'(\rho-3\rho')} + q^{-\rho'} + q^{-(\rho-3\rho')}}.
\end{align*}
Since $0<\eta'<1$, we obtain using \eqref{eq:S3-S4-S'4} and \eqref{eq:S2S3_final}, that
\begin{align*}
  \frac{1}{RS} \sum_{R_1 < r < R} \sum_{1\leq s < S} S_2(r,s) \ll \nu^{3+\omega(q)} q^{\nu} \rb{q^{ - \eta'(\rho-3\rho')} + q^{-\rho'} + q^{\frac{1}{2}(8\lambda-9\mu+8\rho')}}.
\end{align*}
We recall
by \eqref{eq:definition-S} that
\begin{math}
  S = q^{2\rho'}
\end{math}
and
by \eqref{eq:mu-lambda-rho} that
$\mu=\nu-2\rho$, 
$\lambda = \nu+2\rho$
and insert the estimation from above in \eqref{eq:S0-S2}:
\begin{align*}
  \abs{S_0}^4 \ll q^{4\nu-2\rho'} + q^{4\nu-2\rho} + \nu^{3+\omega(q)}  q^{4\nu} \rb{q^{-\eta'(\rho-3\rho')} + q^{-\rho'} + q^{-\frac{\nu}{2}+17\rho+4\rho'} }.
\end{align*}
For $\rho'=\floor{\nu/146}$ 
and $\rho=4\rho'$, we obtain 
\begin{align*} 
  \abs{S_0} \ll \nu^{(3+\omega(q))/4} q^{\nu - \frac{\eta'\rho'}{4}} \ll N^{1-\eta_1},
\end{align*}
for all $\eta_1 < \eta'/584$.
Therefore we have seen that Proposition~\ref{Pro2}
implies the case $K \not \equiv 0 (\bmod 1)$ of Theorem~\ref{Thexponentialsums}.

\section{Auxiliary Results}\label{cha:auxiliary}
In this last section, we present some auxiliary results which are used in \Cref{cha:proof}, to prove the main theorem. 
For this proof, it is crucial to approximate characteristic functions of the intervals $[0,\alpha) \bmod 1$ where $0\leq \alpha < 1$ by trigonometric polynomials. 
This is done by using Vaaler's method - see Section \ref{sec:vaaler}. As we deal with exponential sums we also use 
a generalization of Van-der-Corput's inequality which we have already seen in Section \ref{sec:vandercorput}. In Section \ref{sec:sumgeometric}, we acquire 
some results dealing with sums of geometric series which we use to bound linear exponential sums. Section \ref{sec:gauss} is dedicated to one 
classic result on Gauss sums and allows us to find appropriate bounds on the occurring quadratic exponential sums in \Cref{cha:proof}. 
The last part of this section deals with carry propagation. We find a quantitative statement that carry propagation along 
several digits is rare, i.e. exponentially decreasing.

We would like to note that all these auxiliary results have already been presented in \cite{drmotaMauduitRivat2014}.

\subsection{Sums of geometric series}\label{sec:sumgeometric}

We will often make use of the following upper bound for geometric
series with ratio $\e(\xi), \xi \in \R$ and $L_1,L_2\in\Z$, $L_1 \leq L_2$:
\begin{align}\label{eq:estimate-geometric-series}
  \abs{\sum_{L_1<\ell \leq L_2} \e(\ell \xi)} \leq \min(L_2-L_1,\abs{\sin\pi\xi}^{-1}),
\end{align}
which is obtained from the formula for finite geometric series.

The following results allow us to find useful estimates for special double and triple sums involving geometric series.

\begin{lemma}\label{lemma:sum_inverse_sinus}
 Let $(a,m)\in\Z^2$ with $m\ge 1$, $\delta=\gcd(a,m)$ and $b\in\R$. 
 For any real number $U>0$, we have
\begin{align}\label{eq:sum_inverse_sinus}
  \sum_{0\le n< m} \min \rb{ U, \abs{\sin\rb{ \pi \tfrac{an+b}{m}}}^{-1}} \leq 
      \delta \min\rb{U, \abs{\sin \rb{\pi\tfrac{\delta\, \norm{b/\delta}}{m}}}^{-1}} + \frac{2\, m}{\pi} \log (2\, m).
\end{align}
\end{lemma}
\begin{proof}
  This is \cite[Lemma 6]{drmotaMauduitRivat2014}.
\end{proof}

\begin{lemma}\label{le:sum_sum_sin}
  Let  $m\geq 1$ and $A\geq 1$ be integers and $b\in\R$.
  For any real number $U>0$, we have
  \begin{align}\label{eq:double-sum-min}
    \frac{1}{A} \sum_{1\leq a \leq A} \sum_{0\leq n < m} \min\rb{U, \abs{\sin \rb{\pi \tfrac{ a n + b}{m}}}^{-1}}
        \ll \tau(m) \ U + m\log m
  \end{align}
  and, if $\abs{b}\leq \frac{1}{2}$, we have an even sharper bound
  \begin{align}\label{eq:double-sum-min-sharp}
  \begin{split}
    &\frac{1}{A} \sum_{1\leq a \leq A} \sum_{0\leq n < m} \min\rb{U, \abs{\sin \rb{\pi \tfrac{ a n + b}{m}}}^{-1}}\\
        &\qquad \ll \tau(m) \min\rb{ U, \abs{\sin \rb{\pi \tfrac{b}{m}}}^{-1} } + m\log m,
  \end{split}
  \end{align}
  where $\tau(m)$ denotes the number of divisors of $m$.
\end{lemma}
\begin{proof}
  See \cite{drmotaMauduitRivat2014}.
\end{proof}

\subsection{Gauss sums}\label{sec:gauss}
In the proof of the main theorem, we will meet quadratic exponential sums. 
We first consider Gauss sums $\G(a,b;m)$ which are defined by:
\begin{align*}
	\G(a,b;m) := \sum_{n = 0}^{m-1} \e\rb{\frac{an^2+bn}{m}}.
\end{align*}
In this section, we recall one classic result on Gauss sums, namely Theorem~\ref{theorem:complete-gauss-sum}.
\begin{theorem}\label{theorem:complete-gauss-sum}
  For all $(a, b, m) \in \Z^3$ with $m\geq 1$,
  \begin{align} \label{eq:complete-gauss-sum}
    \abs{\sum_{n=0}^{m-1} \e\rb{\tfrac{an^2+bn}{m}}} \leq \sqrt{2m\gcd(a,m)}
  \end{align}
holds.
\end{theorem}
\begin{proof}
  This form was for example obained in \cite[Proposition 2]{mauduitRivat2009}.
\end{proof}

Consequently we obtain the following result for incomplete quadratic Gauss sums.
\begin{lemma}\label{lemma:incomplete-gauss-sum}
  For all $(a, b, m, N ,n_0) \in \Z^5$ with $m\geq 1$ and $N\geq 0$,
  we have
  \begin{align} \label{eq:incomplete-gauss-sum}
    \abs{\sum_{n=n_0+1}^{n_0+N} \e\rb{\tfrac{an^2+bn}{m}}} \leq \rb{ \tfrac{N}{m} + 1 + \tfrac{2}{\pi} \log\tfrac{2m}{\pi} } \sqrt{2m\gcd(a,m)}.
  \end{align}
\end{lemma}
\begin{proof}
  This is Lemma 9 of \cite{drmotaMauduitRivat2014}.
\end{proof}

\subsection{Carry Lemmas}\label{sec:carry_2}

As mentioned before, we want to find a quantitative statement on how rare carry propagation along several digits is.

\begin{lemma}\label{Lecarry0}
Let $(\nu,\lambda,\rho)\in\N^3$ such that
$\nu+\rho \leq \lambda \leq 2\nu$.
For any integer $r$ with $0\leq r\leq q^{\rho}$, the
number of integers $n < q^{\nu}$  for which there exists an integer
$j\ge \lambda$ 
with $\varepsilon_j((n+r)^2) \ne \varepsilon_j(n^2)$ is
$\ll q^{2\nu+\rho-\lambda}$.
Hence, we find for any block-additive function $b$, that the number of integers $n < q^\nu$  with
\begin{align*}
b_{\lambda-m+1}((n+r)^2) - b_{\lambda-m+1}(n^2) \ne b((n+r)^2) - b(n^2)
\end{align*}
is also $\ll q^{2\nu+\rho-\lambda}$.
\end{lemma}
\begin{proof}
  A proof for the Thue-Morse sequence can be found in \cite{drmotaMauduitRivat2014} and it is easy to adapt it for this more general case.
\end{proof}

The next lemma helps to replace quadratic exponential sums depending only on few digits.

\begin{lemma}\label{Lecarry1}
Let $(\lambda,\mu,\nu,\rho')\in\N^4$ 
such that $0 < \mu < \nu < \lambda$, $2\rho' \le \mu \le \nu - \rho'$ and $\lambda-\nu \le 2(\mu-\rho')$and set 
$\mu' = \mu - \rho'$.
For integers $n<q^\nu$, $s\ge 1$ 
and $1\le r\le q^{(\lambda-\nu)/2}$ we set
\begin{align}
  n^2 &\equiv u_1 q^{\mu'} + w_1 (\bmod q^{\lambda+m-1}) & (0\le w_1 < q^{\mu'},\ 0\le u_1 <  q^{\lambda +m-1- \mu + \rho'})  \nonumber \\
  (n+r)^2 &\equiv u_2 q^{\mu'} + w_2 (\bmod q^{\lambda+m-1}) \label{equ1u2u3} 
    & (0\le w_2 < q^{\mu'},\ 0\le u_2<  q^{\lambda +m-1- \mu + \rho'})\\
  2n &\equiv u_3 q^{\mu'} + w_3 (\bmod q^{\lambda+m-1}) &  (0\le w_3 < q^{\mu'},\ 0\le u_3 < q^{\nu+1 - \mu + \rho'})  \nonumber \\
  2s q^{m-1} n &\equiv v (\bmod q^{\lambda-\mu+m-1}),  & (0\le v < q^{\lambda-\mu+m-1}) \nonumber 
\end{align}
where the integers $u_1=u_1(n)$, $u_2=u_2(n)$, $u_3=u_3(n)$, 
$v=v(n)$, $w_1=w_1(n)$, $w_2=w_2(n)$ and $w_3=w_3(n)$ satisfy
the above conditions. 
Then for any integer $\ell\ge 1$ the number of integers $n < q^\nu$ for which
one of the following conditions 
\begin{align}
  b_{\mu,\lambda}((n+\ell)^2) &\ne b_{\rho',\lambda-\mu+\rho'}(u_1+\ell u_3) \nonumber \\
  b_{\mu,\lambda}((n+\ell+sq^{\mu+m-1})^2))
  &\ne b_{\rho',\lambda-\mu+\rho'}(u_1+\ell u_3+ v q^{\rho'} + 2 \ell s q^{m-1} q^{\rho'}) \label{equ1u2u3-2} \\
  b_{\mu,\lambda}((n+r+\ell)^2) &\ne b_{\rho',\lambda-\mu+\rho'}(u_2+\ell u_3)\nonumber\\
  b_{\mu,\lambda}((n+r+\ell+sq^{\mu+m-1})^2))
  &\ne b_{\rho',\lambda-\mu+\rho'}(u_2+\ell u_3+ v q^{\rho'} + 2 (\ell+r) s q^{m-1} q^{\rho'}) \nonumber
\end{align}
is satisfied is $\ll q^{\nu-\rho'}$.
\end{lemma}

\begin{proof}
A proof for the sum of digits function in base $2$ can be found in \cite{drmotaMauduitRivat2014} and it is straight forward to adapt it to fit this more general case.
\end{proof}

\subsection{Van-der-Corput's inequality}\label{sec:vandercorput}

The following lemma is a generalization of Van-der-Corput's inequality.
\begin{lemma}[\cite{mauduitRivat2009}]\label{lemma:van-der-corput}
  For all complex numbers $z_1,\ldots,z_N$ 
  and all integers $Q\geq 1$ and  $R\geq 1$, we have
  \begin{align}\label{eq:van-der-corput}
    \abs{\sum_{n=1}^{N-1} z_n}^2 \le \frac{N+QR-Q}{R} \rb{\sum_{n=1}^{N-1} |z_n|^2 + 2\ \sum_{r=1}^{R-1}\rb{1-\frac{r}{R}} \ \sum_{n=1}^{N-Qr-1}\Re\rb{z_{n+Qr} \overline{z_n}}}
  \end{align}
  where $\Re(z)$ denotes the real part of $z\in\C$.
\end{lemma}

\subsection{Vaaler's method}\label{sec:vaaler}

The following theorem is a classical method to detect real numbers in an
interval modulo $1$ by means of exponential sums developed by Vaaler~\cite{vaaler}. 
For $\alpha\in\R$ with $0\leq \alpha<1$, we denote by $\chi_\alpha$ the
characteristic function of the interval $[0,\alpha)$ modulo $1$:
\begin{align} \label{eq:definition-chi}
  \chi_\alpha(x)=\floor{x}-\floor{x-\alpha}.
\end{align}

The following theorem is a consequence of the mentioned paper by Vaaler. 
The presented form was first published by Mauduit and Rivat~\cite{mauduit_rivat_rs}.

\begin{theorem}\label{th:vaaler}
For all $\alpha\in\R$ with $0\leq \alpha<1$ 
and all integer $H\geq 1$, 
there exist real-valued
trigonometric polynomials $A_{\alpha,H}(x)$ and $B_{\alpha,H}(x)$ 
such that for all $x\in\R$
\begin{align}\label{eq:vaaler-approximation}
  \abs{ \chi_\alpha(x) - A_{\alpha,H}(x) } \leq B_{\alpha,H}(x).
\end{align}
The trigonometric polynomials are defined by
\begin{align}\label{eq:definition-A-B}
  A_{\alpha,H}(x) = \sum_{\abs{h}\leq H}  a_h(\alpha,H) \e(h x),\
  B_{\alpha,H}(x) = \sum_{\abs{h}\leq H} b_h(\alpha,H) \e(h x),
\end{align}
with coefficients $a_h(\alpha,H)$ and $b_h(\alpha,H)$ satisfying
\begin{align}\label{eq:vaaler-coef-majoration}
  a_0(\alpha,H) = \alpha,\ 
  \abs{a_h(\alpha,H)} \leq \min\rb{\alpha,\tfrac{1}{\pi\abs{h}}},\
  \abs{b_h(\alpha,H)} \leq \tfrac{1}{H+1}.
\end{align}
\end{theorem}

Using this method we can detect points in a $d$-dimensional box (modulo $1$):
\begin{lemma}\label{lemma:better-koksma1}
  For $(\alpha_1,\ldots, \alpha_d) \in [0,1)^d$ 
  and $(H_1,\ldots,H_d)\in\N^d$ 
  with $H_1\geq 1$,\ldots, $H_d\geq 1$,
  we have for all $(x_1,\ldots,x_d)\in\R^d$
  \begin{align} \label{eq:better-koksma1}
    \abs{\prod_{j=1}^d \chi_{\alpha_j}(x_j) - \prod_{j=1}^d A_{\alpha_j,H_j}(x_j)}
      \leq \sum_{\emptyset \neq J \subseteq \{1,\ldots,d\}} \prod_{j\not\in J} \chi_{\alpha_j}(x_j) \prod_{j\in J} B_{\alpha_j,H_j}(x_j)
  \end{align}
  where $A_{\alpha,H}(.)$ and $B_{\alpha,H}(.)$ are the real valued
  trigonometric polynomials defined by \eqref{eq:definition-A-B}.
\end{lemma}
\begin{proof}
  See again \cite{mauduit_rivat_rs}.
\end{proof}

Let $(U_1,\ldots,U_d)\in\N^d$ with $U_1\geq 1$,\ldots,$U_d\geq 1$ and
define $\alpha_1=1/U_1$,\ldots,$\alpha_d=1/U_d$.
For $j=1,\ldots,d$ and $x\in\R$ we have
\begin{align}\label{eq:chi-partition}
  \sum_{0\leq u_j < U_j} \chi_{\alpha_j}\rb{x-\frac{u_j}{U_j}} = 1.
\end{align}
Let $N\in\N$ with $N\ge 1$, $f:\{1,\ldots,N\} \to \R^d$
and $g:\{1,\ldots,N\} \to \C$ such that $\abs{g}\leq 1$.
If $f=(f_1,\ldots,f_d)$,
we can express the sum
\begin{align*}
  S= \sum_{n=1}^N g(n) 
\end{align*}
as
\begin{align*}
  S = \sum_{n=1}^N g(n) \sum_{0\leq u_1 < U_1} \chi_{\alpha_1}\rb{f_1(n)-\frac{u_1}{U_1}} \cdots \sum_{0\leq u_d < U_d} \chi_{\alpha_d}\rb{f_d(n)-\frac{u_d}{U_d}}  .
\end{align*}
We now define $(H_1,\ldots,H_d)\in\N^d$ with $H_1\geq 1$,\ldots, $H_d\geq 1$,
\begin{align*}
  \widetilde{S} = \sum_{n=1}^N g(n) \sum_{0\leq u_1 < U_1} A_{\alpha_1,H_1}\rb{f_1(n)-\frac{u_1}{U_1}}  
      \cdots \sum_{0\leq u_d < U_d} A_{\alpha_d,H_d}\rb{f_d(n)-\frac{u_d}{U_d}}.
\end{align*}

\begin{lemma}
  With the notations from above, we have
  \begin{align}
  \label{eqS-S} \abs{S-\widetilde{S}} &\leq \sum_{\ell=1}^{d-1} \sum_{1\leq j_1<\cdots<j_\ell} \frac{U_{j_1}\cdots U_{j_\ell}}{H_{j_1}\cdots H_{j_\ell}}
      \sum_{\abs{h_{j_1}}\leq H_{j_1}/U_{j_1}} \cdots \sum_{\abs{h_{j_\ell}}\leq H_{j_\ell}/U_{j_\ell}}\\
    &\qquad \abs{\sum_{n=1}^N \e\rb{h_{j_1} U_{j_1} f_{j_1}(n) + \cdots + h_{j_\ell} U_{j_\ell} f_{j_\ell}(n)}}.
\end{align}
\end{lemma}
\begin{proof}
  See again \cite{mauduit_rivat_rs}.
\end{proof}

\bibliographystyle{abbrv}
 \bibliography{bibliography}

\begin{thebibliography}{10}

\bibitem{AlloucheShallit}
J.-P. {Allouche} and J.~{Shallit}.
\newblock {\em {Automatic sequences. Theory, applications, generalizations.}}
\newblock Cambridge: Cambridge University Press, 2003.

\bibitem{normal_polynomial}
V.~Becher, P.~A. Heiber, and T.~A. Slaman.
\newblock A polynomial-time algorithm for computing absolutely normal numbers.
\newblock {\em Inform. and Comput.}, 232:1--9, 2013.

\bibitem{bellmanShapiro}
R.~Bellman and H.~N. Shapiro.
\newblock On a problem in additive number theory.
\newblock {\em Ann. of Math. (2)}, 49:333--340, 1948.

\bibitem{normal}
Y.~Bugeaud.
\newblock {\em Distribution modulo one and {D}iophantine approximation}, volume
  193 of {\em Cambridge Tracts in Mathematics}.
\newblock Cambridge University Press, Cambridge, 2012.

\bibitem{Cateland}
E.~Cateland.
\newblock Suites digitales et suites k-r{\'e}guli{\`e}res.
\newblock PhD Thesis, Universit{\'e} Bordeaux I, 1992.

\bibitem{drmotaMauduitRivat2014}
M.~Drmota, C.~Mauduit, and J.~Rivat.
\newblock The {T}hue-{M}orse sequence along squares is normal.
\newblock \url{http://www.dmg.tuwien.ac.at/drmota/alongsquares.pdf}, 2013.
\newblock [Online; first accessed 26.2.2014].

\bibitem{seq_complex}
S.~Ferenczi.
\newblock Complexity of sequences and dynamical systems.
\newblock {\em Discrete Math.}, 206(1-3):145--154, 1999.
\newblock Combinatorics and number theory (Tiruchirappalli, 1996).

\bibitem{subst}
N.~P. Fogg.
\newblock {\em Substitutions in dynamics, arithmetics and combinatorics},
  volume 1794 of {\em Lecture Notes in Mathematics}.
\newblock Springer-Verlag, Berlin, 2002.
\newblock Edited by V. Berth{\'e}, S. Ferenczi, C. Mauduit and A. Siegel.

\bibitem{gelfond1968}
A.~O. Gelfond.
\newblock Sur les nombres qui ont des propri\'et\'es additives et
  multiplicatives donn\'ees.
\newblock {\em Acta Arith.}, 13:259--265, 1967/1968.

\bibitem{dynamics}
P.~Kurka.
\newblock {\em Topological and symbolic dynamics}, volume~11 of {\em Cours
  Sp\'ecialis\'es [Specialized Courses]}.
\newblock Soci\'et\'e Math\'ematique de France, Paris, 2003.

\bibitem{Levin_normal}
M.~B. Levin.
\newblock Absolutely normal numbers.
\newblock {\em Vestnik Moskov. Univ. Ser. I Mat. Mekh.}, (1):31--37, 87, 1979.

\bibitem{mauduitRivat2009}
C.~Mauduit and J.~Rivat.
\newblock La somme des chiffres des carr\'es.
\newblock {\em Acta Math.}, 203(1):107--148, 2009.

\bibitem{mauduit_rivat_rs}
C.~Mauduit and J.~Rivat.
\newblock Prime numbers along {R}udin-{S}hapiro sequences.
\newblock {\em J. Eur. Math. Soc. (JEMS)}, 17(10):2595--2642, 2015.

\bibitem{france}
M.~Mend\`es~France.
\newblock Nombres normaux. {A}pplications aux fonctions pseudo-al\'eatoires.
\newblock {\em J. Analyse Math.}, 20:1--56, 1967.

\bibitem{moshe}
Y.~Moshe.
\newblock On the subword complexity of {T}hue-{M}orse polynomial extractions.
\newblock {\em Theoret. Comput. Sci.}, 389(1-2):318--329, 2007.

\bibitem{queffelec}
M.~Queff\'elec.
\newblock {\em Substitution dynamical systems---spectral analysis}, volume 1294
  of {\em Lecture Notes in Mathematics}.
\newblock Springer-Verlag, Berlin, second edition, 2010.

\bibitem{scheerer_normal}
A.-M. Scheerer.
\newblock Computable absolutely normal numbers and discrepancies.
\newblock {\em Math. Comp.}, 86(308):2911--2926, 2017.

\bibitem{Schmidt_normal}
W.~M. Schmidt.
\newblock \"uber die {N}ormalit\"at von {Z}ahlen zu verschiedenen {B}asen.
\newblock {\em Acta Arith.}, 7:299--309, 1961/1962.

\bibitem{Sierpinski_normal}
W.~Sierpinski.
\newblock D\'emonstration \'el\'ementaire du th\'eor\`eme de {M}. {B}orel sur
  les nombres absolument normaux et d\'etermination effective d'une tel nombre.
\newblock {\em Bull. Soc. Math. France}, 45:125--132, 1917.

\bibitem{Turing_normal}
A.~Turing and P.~Saunders.
\newblock {\em Collected Works of A. M. Turing}.
\newblock Number Bd. 4 in Advances in Psychology. North-Holland, 1992.

\bibitem{vaaler}
J.~D. Vaaler.
\newblock Some extremal functions in {F}ourier analysis.
\newblock {\em Bull. Amer. Math. Soc. (N.S.)}, 12(2):183--216, 1985.

\end{thebibliography}

\end{document}